\documentclass[10pt]{amsart}

\title[On the irrationality of cubic fourfolds]{On the irrationality of cubic fourfolds}
\author[Gu\'er\'e]{J\'er\'emy Gu\'er\'e}
\address{Univ. Grenoble Alpes, CNRS, IF, 38000 Grenoble, France}
\email{jeremy.guere@univ-grenoble-alpes.fr}

\usepackage{amsrefs}
\usepackage{amsfonts}
\usepackage{amsmath}
\usepackage{amssymb}
\usepackage{amscd}
\usepackage{array}
\usepackage{subfigure}
\usepackage{tikz}
\usetikzlibrary{arrows}
\usepackage{pifont}
\usepackage{wasysym}
\usepackage{tikz-cd}

\usepackage[all]{xy}

\allowdisplaybreaks[3]

\newcommand{\coeur}{\text{\scalebox{1.16}{\ding{170}}}}

\newcommand{\id}{\mathrm{id}}

\newcommand{\pt}{\mathrm{pt}}

\newcommand{\EE}{\mathbb{E}}
\newcommand{\cN}{\mathcal{N}}

\newcommand{\un}{\textbf{1}}
\newcommand{\vir}{\mathrm{vir}}
\newcommand{\rk}{\mathrm{rk}}

\newcommand{\ext}{\mathrm{ext}}

\newcommand{\ci}{\mathrm{i}}
\newcommand{\PP}{\mathbb{P}}
\newcommand{\CC}{\mathbb{C}}
\newcommand{\ZZ}{\mathbb{Z}}
\newcommand{\NN}{\mathbb{N}}

\newcommand{\QQ}{\mathbb{Q}}

\newcommand{\ev}{\mathrm{ev}}

\renewcommand{\(}{\left(}
\renewcommand{\)}{\right)}

\newcommand{\cF}{\mathcal{F}}

\newcommand{\cM}{\mathcal{M}}

\theoremstyle{plain}

\newtheorem{thm}{Theorem}
\newtheorem*{thm*}{Theorem}
\newtheorem{pro}[thm]{Proposition}
\newtheorem{lem}[thm]{Lemma}
\newtheorem*{lem*}{Lemma}

\newtheorem{cor}[thm]{Corollary}

\newtheorem*{conv*}{Convention}

\theoremstyle{definition}

\newtheorem{dfn}[thm]{Definition}
\newtheorem{nt}[thm]{Notation}

\newtheorem{assumption}[thm]{Assumption}
\newtheorem{rem}[thm]{Remark}
\newtheorem*{rem*}{Remark}
\newtheorem*{rems*}{Remarks}
\newtheorem*{warn*}{Warning}
\newtheorem{exa}[thm]{Example}
\newtheorem*{exa*}{Example}
\newtheorem*{exait*}{\rm \em Example}
\newtheorem*{exadefit*}{\rm \em Example/Definition}
\newtheorem*{cla*}{\rm \em Claim}

\newtheorem*{dfn*}{Definition}

\DeclareMathOperator{\vdim}{vdim}

\usepackage{amsxtra}

\newcommand{\Eu}{\mathrm{Eu}}
\newcommand{\Hdg}{\mathrm{Hdg}}
\newcommand{\Sp}{\mathrm{Sp}}

\def\<{\left\langle}
\def\>{\right\rangle}

\begin{document}
\begin{abstract}
Following the work of Katzarkov--Kontsevich--Pantev--Yu \cite{KKPY} concerning the irrationality of the very general complex cubic fourfold, we prove the following: for every rational smooth complex cubic fourfold, the primitive cohomology is isomorphic as a rational Hodge structure to the (twisted) middle cohomology of a projective K3 surface.
\end{abstract}

\maketitle

\tableofcontents

\setcounter{section}{-1}

\section{Introduction}
In \cite{KKPY}, Katzarkov--Kontsevich--Pantev--Yu develop a framework for using Gromov--Witten theory to study birational geometry.
Since Gromov--Witten theory is deformation
invariant, its utility in distinguishing between rational and irrational smooth cubic fourfolds may initially seem surprising.
However, by combining it with Hodge theory, one can define properties of smooth projective complex varieties that remain invariant under blow-ups at smooth centers, at least for complex dimensions up to four.
This approach is underpinned by \cite[Theorem 1.1]{Iritani}, which describes the behaviour of quantum cohomology under such blow-ups.

In this paper, we adapt the methods of \cite{KKPY} in order to prove the following theorem, which addresses part of Hassett's conjecture (see e.g.~\cite{Hassett}).

\begin{thm*}[see Theorem \ref{rat cubic K3}]
If $X$ is a rational smooth complex cubic fourfold, then there exists a projective K3 surface $\Sigma$ and an isomorphism of (rational) Hodge structures
$$H^4(X,\QQ)_\mathrm{primitive} \simeq H^2(\Sigma,\QQ)(-1).$$
\end{thm*}

Our proof relies on a specific property of quantum cohomology that remains invariant under blow-ups.
To clarify the distinction between our approach and that of \cite{KKPY}, we denote our property by $\coeur$ and their original property by $\clubsuit$ (see Definition \ref{key}).
Additionally, for the reader's convenience, we recall the relevant material from \cite{KKPY} and provide an outline of their main result.

\begin{rem*}
For simplicity's sake we have chosen not to discuss the exact definition of an `Hodge atom', since it is not directly necessary for our proof (and would have involved a long digression).
We refer the interested reader to \cite{KKPY} for further details.
Nevertheless, the evaluation maps (see Definition \ref{eval}) are closely related to the Hodge atom construction.
For instance, when an evaluation map $\ev$ is taken to be `generic' (specifically when the cardinality of $\Sp^X_\ev$ is maximal), then the generalized eigenspace $E^X_{\ev,\alpha}$ attached to an eigenvalue $\alpha \in \Sp^X_\ev$ (see Definition \ref{Spectrum}) is sometimes called a `coarse (Hodge) atom'.
\end{rem*}

\begin{warn*}
It is not enough to focus only on \textit{generic} evaluation maps. Indeed, we will have to make a distinction in Definition \ref{key} between a Property $\diamondsuit$ of coarse atoms (i.e.~that is only valid on generic evaluation maps) and a stronger Property $\diamondsuit^{\mathrm{strong}}$ that is valid on all evaluation maps.
Specifically, in Corollary \ref{main}, we require Property $\diamondsuit^{\mathrm{strong}}$ to hold for all blow-up centers.
The reader should keep in mind this subtlety, especially for the day when one tries to apply the atom theory to birationality questions in dimension at least five.
\end{warn*}


\noindent
\textbf{Acknowledgement.}
I would like to thank the organizers of the reading group on \cite{KKPY} in Paris, especially Emanuele Macr\`i and Claire Voisin, where this project originated.
I am also grateful to Alessandro Chiodo, Hsueh-Yung Lin, Laurent Manivel, Vladimiro Benedetti, Nicolas Perrin, Michel Brion, and Catriona Maclean for many fruitful discussions and their assistance regarding this work.
Special thanks are due to Hiroshi Iritani for his clarifications on \cite{Iritani} and for his helpful corrections to this manuscript.
Finally, I would like to thank my wife and children for their patience and support during the completion of this work.

\section{Quantum cohomology}
Throughout this paper, we consider only smooth projective complex varieties.
Let $X$ be such a variety.
In this section, we provide basic definitions regarding Hodge structures, followed by a description of Gromov--Witten theory and quantum multiplication.

~~

\subsection{Hodge structures}
\begin{dfn}[Hodge structure]
Let $V_\QQ$ be a $\QQ$-vector space and $n \in \ZZ$.
A \textit{Hodge structure} on $V_\QQ$ of weight $n$ is a decomposition
$$V_\QQ \otimes_\QQ \CC = \bigoplus_{\substack{p,q \in \ZZ \\ p+q=n}} V^{p,q}$$
such that $V^{p,q} = \overline{V^{q,p}}$ under complex conjugation.
A \textit{morphism of Hodge structures} $f \colon V_\QQ \to W_\QQ$ is a morphism of $\QQ$-vector spaces whose complexification satisfies
$$f_\CC\(V^{p,q}\) \subset W^{p,q}.$$
\end{dfn}

\begin{exa*}[Deligne twist]
Denote by $\QQ(1) := V_\QQ = 2 \pi \ci \QQ \subset \CC$ the one-dimensional $\QQ$-vector space with the decomposition
$$V_\QQ \otimes_\QQ \CC = V^{-1,-1}.$$
This is a Hodge structure of weight $-2$ called the \textit{Deligne twist}.
If $W_\QQ$ is another Hodge structure, we use the notation
$$W_\QQ(n) := W_\QQ \otimes_\QQ \QQ(1)^{\otimes_\QQ n}$$
for the \textit{twisted Hodge structure}.
\end{exa*}

\begin{dfn}[Hodge classes and Hochschild degree]
Let $V_\QQ$ be a $\QQ$-vector space equipped with a Hodge structure of weight $n$. The elements of the subspace
$$V^\Hdg := \left\lbrace \begin{array}{ll}
V_\QQ \cap V^{p,p} & \textrm{if $n=2p$ for $p \in \ZZ$,} \\
0 & \textrm{if $n$ is odd,}
\end{array} \right. $$
are called \textit{Hodge classes}.
For any $k \in \ZZ$, we define the subspace of classes of \textit{Hochschild degree} $k$ as
$$V^{(k)} := \bigoplus_{\substack{p,q \in \ZZ \\ p-q=k}} V^{p,q}.$$
\end{dfn}

\begin{exa*}
The rational cohomology $H^n(X,\QQ)$ of $X$ carries a Hodge structure
$$H^n(X,\CC) = \bigoplus_{\substack{p,q \geq 0 \\ p+q=n}} H^{p,q}(X).$$
We denote by $H^*(X)^\Hdg$ (resp.~$H^{(k)}(X)$) the subspace of Hodge classes (resp.~classes of Hochschild degree $k$).
\end{exa*}

\subsection{Formal power series}
We begin with the definition of the Novikov ring.
\begin{dfn}[Novikov ring]\label{Novikov}
Let $\mathrm{NE}_\NN(X) \subset H_2(X,\ZZ)$ denote the monoid generated by classes of effective curves in $X$.
We write $Q^\beta \in \QQ[\mathrm{NE}_\NN(X)]$ for the element corresponding to the curve $\beta \in \mathrm{NE}_\NN(X)$, and define its degree by
$$\deg \(Q^\beta\) := 2 \int_{\beta} c_1(X).$$
Let $\omega$ be an ample class in $X$. We denote by
$$\QQ[[Q]] := \QQ[[\mathrm{NE}_\NN(X)]]$$
the graded completion of $\QQ[\mathrm{NE}_\NN(X)]$ with respect to the descending chain of submodules
$$I_k = \langle Q^\beta ~|~\int_{\beta}\omega \geq k \rangle.$$
This completion is dependent of the choice of $\omega$; we call $\QQ[[Q]]$ the (rational) \textit{Novikov ring} of $X$.
More generally, for any $\ZZ$-graded module $K$, we define $K[[Q]]$ as the graded completion of $K[Q]$.
Specifically, we first consider the submodule $R_n \subset K[Q]$ of polynomials homogeneous of degree $n$ for each $n \in \ZZ$, complete $R_n$ into $\widehat{R}_n$ with respect to the descending chain above, and then take the direct sum
$$K[[Q]] := \bigoplus_{n \in \ZZ} \widehat{R}_n.$$
In particular, this construction yields a graded $K$-algebra.
\end{dfn}

We fix a graded basis $$\mathcal{H}:=\(\alpha_0,\dotsc,\alpha_h\)$$
of the vector space $H^*(X)^{\Hdg}$, with $\alpha_0:=\un_X$, and introduce a formal variable $T_k$ for each $\alpha_k$, with degree defined by
$$\deg\(T_k\):=2-\deg(\alpha_k).$$
Following Definition \ref{Novikov}, we obtain a graded $\QQ$-algebra
$$R^*(X,\QQ) := \QQ[[Q,T_0,\dotsc,T_h]].$$
Explicitly, an element of $R^*(X,\QQ)$ is a formal power series
$$\sum_{\substack{-D \leq d \leq D \\ N \geq 0}} \sum_{\substack{n_0,\dotsc,n_h \geq 0 \\ \beta \in \mathrm{NE}_\NN(X)\\n_0+\dotsc+n_h+\int_{\beta}\omega=N \\ n_0 \deg\(T_0\) + \dotsb + n_h \deg\(T_h\) + \deg\(Q^\beta\) = d}} c_{n_0,\dotsc,n_h,\beta} ~ T_0^{n_0} \dotsm T_h^{n_h} Q^\beta,$$
with $D \geq 0$ and coefficients $c_{n_0,\dotsc,n_h,\beta} \in \QQ$
such that there are only finitely many non-zero coefficients in the second sum, where $\omega$ is any ample class on $X$.

\begin{dfn}\label{formal Hodge}
For any field extension $\QQ \subset K \subset \CC$, any graded $K$-algebra $R^*$, and any graded $K$-vector space $V^*$, we define
$$V^*_R := \bigoplus_{n \in \ZZ} V^n_R ~,\quad V^n_R := \bigoplus_{\substack{p,q \in \ZZ \\ p+2q=n}} V^p \otimes_K R^{2q},$$
where we restrict $R^*$ to $R^\mathrm{even}$, i.e.~to terms of \textit{even degree}.
This is a free $R^\mathrm{even}$-module and a graded $K$-algebra.
Our primary example is $H^*(X,\QQ)_{R(X,\QQ)}$.

Now, let $R^*$ be a graded $\QQ$-algebra and $V^*$ be a graded $\QQ$-vector space such that each $V^p$ carries a Hodge structure.
For each $q \in \ZZ$, we define a Hodge structure on the $\QQ$-vector space $R^{2q}$ by setting $R^{2q} \otimes_\QQ \CC = V^{q,q}$, i.e.~we take the trivial Hodge structure on $R^{2q}$ and twist it by $(-q)$.
Then, for any $n \in \ZZ$, we define a Hodge structure on the $\QQ$-vector space 
$$V^n_R = \bigoplus_{\substack{p,q \in \ZZ \\ p+2q=n}} V^p \otimes_\QQ R^{2q}$$
by tensoring the Hodge structure of $V^p$ with the twist $(-q)$.
In particular, the subspace of Hodge classes is $V^\Hdg_R = V^\Hdg \otimes_\QQ R$, and the subspace of classes with Hochschild degree $k$ is $V^{(k)}_R =V^{(k)} \otimes_\CC R_\CC$.
\end{dfn}

\begin{dfn}[$K$-span of morphisms of Hodge structures]\label{Kspan}
Let $K$ be a field extension $\QQ \subset K \subset \CC$, $R^*$ be a graded $\QQ$-algebra, and $V^*$ and $W^*$ be two graded $\QQ$-vector spaces equipped with Hodge structures in each degree as in Definition \ref{formal Hodge}.
Assume we have a morphism of free $R_K^\mathrm{even}$-modules
$$f \colon V^{*-2k}_{R_K} \to W^*_{R_K}$$
that is homogeneous of some even degree $2k \in \ZZ$ and defined over $K$.
We say that $f$ is in the \textit{$K$-span of morphisms of Hodge structures} if there exist finitely many morphisms of free $R^\mathrm{even}$-modules
$$f_1, \dotsc, f_N \colon V^{*-2k}_R(-k) \to W^*_R$$
which are also morphisms of Hodge structures, such that
$$f = s_1 f_1+\dotsb + s_N f_N$$
for some scalars $s_1, \dotsc, s_N \in K$ that are linearly independent over $\QQ$.
\end{dfn}

\subsection{Quantum product}
We briefly recall several properties of Gromov--Witten invariants.
For any non-negative integer $n$ and any effective curve class $\beta \in \mathrm{NE}_\NN(X)$ such that $\beta \neq 0$ or $n \geq 3$, there exists a proper Deligne--Mumford stack $\cM_{0,n}(X,\beta)$ parametrizing stable maps $f \colon C \to X$ from an $n$-marked genus-zero nodal curve $C$ to $X$ such that $f_*[C] = \beta \in H_2(X,\ZZ)$.
Moreover, there exists an algebraic cycle
$$\left[ \cM_{0,n}(X,\beta)\right]^\vir \in H_{2\vdim}(\cM_{0,n}(X,\beta),\QQ),$$
called the \textit{virtual fundamental cycle}, of pure complex dimension
$$\vdim := \dim_\CC(X)-3+n+\int_{\beta}c_1(X) \in \ZZ.$$
There is also an evaluation map
$$\ev^* \colon H^*(X,\QQ)^{\otimes n} \to H^*(\cM_{0,n}(X,\beta),\QQ).$$
For any cohomology classes $\gamma_1, \dotsc, \gamma_n \in H^*(X,\QQ)$, we define the \textit{correlator}
$$\langle \gamma_1, \dotsc, \gamma_n \rangle_{0,n,\beta}^X := \int_{\left[ \cM_{0,n}(X,\beta)\right]^\vir} \ev^*\(\gamma_1\otimes \dotsm \otimes \gamma_n\) \in \QQ.$$
In particular, if the correlator is non-zero, the \textit{dimension equation} must hold:
$$\sum_{i=1}^n \deg\(\gamma_i\) = 2 \vdim.$$
Correlators are symmetric under permutations of the markings.
Furthermore, we will repeatedly use the following identities, known as the \textit{string and divisor equations}.
Whenever $n \geq 3$ or $\beta \neq 0$:
\begin{eqnarray*}
\langle \gamma_1, \dotsc, \gamma_n, \un_X \rangle_{0,n+1,\beta}^X & = & 0, \\
\langle \gamma_1, \dotsc, \gamma_n, \gamma_{n+1} \rangle_{0,n+1,\beta}^X & = & \int_{\beta} \gamma_{n+1} \cdot \langle \gamma_1, \dotsc, \gamma_n \rangle_{0,n,\beta}^X \quad \textrm{if $\gamma_{n+1} \in H^2(X,\QQ)$.}
\end{eqnarray*}

Next, we define \textit{quantum multiplication}.
To this end, we fix a graded basis
$$\mathcal{B}:=(\phi_0,\dotsc,\phi_m)$$
of the cohomology $H^*(X,\QQ)$, with $\phi_0:=\un_X \in H^0(X,\QQ)$.
We denote by $\(\phi^k\)$ the (dual) basis defined by the condition
$$\int_X \phi_i \cup \phi^k = \delta_{i,k} ~,~~\forall ~0 \leq i,k \leq m.$$

\begin{dfn}[Quantum product]
Let
$$\tau := \sum_{k=0}^h T_k \alpha_k \in H^*(X)^\Hdg[[T_0,\dotsc,T_h]] \subset H^*(X,\QQ)_{R(X,\QQ)}.$$
We define the \textit{quantum product} by
$$\phi_i \star_\tau \phi_j = \sum_{\substack{0 \leq k \leq m, n \geq 0 \\ \beta \in \mathrm{NE}_\NN(X)}} \langle \phi_i,\phi_j,\phi_k,\tau,\dotsc,\tau\rangle_{0,3+n,\beta}^X \cfrac{Q^\beta}{n!} \phi^k \in H^*(X,\QQ) \otimes_\QQ R^*(X,\QQ),$$
using the genus-zero Gromov--Witten invariants of $X$.
\end{dfn}


\begin{rem}[Change of variables]\label{change of variables}
Given a field extension $\QQ \subset K' \subset \CC$, let
$$f\(Q^\beta\) := \mu Q^\beta,~f(T_0),~\dotsc,~f(T_h) \in R^*(X,K')~,~~\textrm{where }\mu \in K',$$
such that the function $f$ is homogeneous of degree $0$, i.e.~we have $\deg(f(u))=\deg(u)$ for any $u \in \left\lbrace Q^\beta, T_0, \dotsc, T_h\right\rbrace$. For any $0 \leq k \leq h$, we write
$$f(T_k) =: \lambda_k + g(T_k),$$
where the formal power series $g(T_k)$ has no constant term and $\lambda_k \in K'$.
By the homogeneity of $f$, note that $\deg\(\alpha_k\) \neq 2 \implies \lambda_k=0$.
We define
$$g \( Q^\beta \) = e^{\sum_{k=0}^h \lambda_k \int_{\beta} \alpha_k} f\(Q^\beta\),$$
which belongs to $R^*(X,K)$ for some field extension $K' \subset K \subset \CC$ containing the value of this exponential.
Since the formal power series
$g\(Q^\beta\), g(T_0), \dotsc, g(T_h)$
have no constant terms, we obtain a well-defined (continuous) morphism
$$g \colon R^*(X,K) \to R^*(X,K)$$
of graded $K$-algebras given by
$$\sum_{\substack{n_0,\dotsc,n_h \geq 0 \\ \beta \in \mathrm{NE}_\NN(X)}} c_{n_0,\dotsc,n_h,\beta} ~ T_0^{n_0} \dotsm T_h^{n_h} Q^\beta \mapsto \sum_{\substack{n_0,\dotsc,n_h \geq 0 \\ \beta \in \mathrm{NE}_\NN(X)}} c_{n_0,\dotsc,n_h,\beta} ~ g(T_0)^{n_0} \dotsm g(T_h)^{n_h} g\(Q^\beta\).$$
\end{rem}

\begin{rem}\label{formal quantum product}
Throughout this paper, we consider
$$\tau := \sum_{k=0}^h f_k \alpha_k \in H^2(X)^\Hdg_{R(X,K')} \subset H^2(X,K')_{R(X,K')},$$
i.e.~homogeneous elements of degree $2$, where $\QQ \subset K' \subset \CC$ is a field extension.
Then the function $f$ mapping $T_k$ to $f_k$ and $Q$ to $Q$ is homogeneous of degree zero as in Remark \ref{change of variables}.
Thus we obtain a morphism of graded $K$-algebras
$$g \colon R^*(X,K) \to R^*(X,K)~,~~T_0 \alpha_0 + \dotsb + T_h \alpha_h \mapsto \tau,$$
for some field extension $K' \subset K \subset \CC$.
Consequently, by the divisor equation, we can extend the quantum product as
$$\phi_i \star_\tau \phi_j := g\(\phi_i \star_{T_0 \alpha_0+\dotsb+T_h \alpha_h} \phi_j\).$$
\end{rem}

\begin{dfn}[Euler vector field]
For any element
$$\tau := \sum_{k=0}^h f_k \alpha_k \in H^2(X)^\Hdg_{R(X,\QQ)},$$
that is homogeneous of degree $2$, we define the \textit{Euler vector field} as
$$\Eu_\tau = c_1(X) + \sum_{k=0}^h \(1-\cfrac{\deg \alpha_k}{2}\) f_k \alpha_k \in H^2(X)^\Hdg_{R(X,\QQ)}.$$
\end{dfn}

\begin{dfn}[Endomorphism]\label{matrix}
For any field extension $\QQ \subset K' \subset \CC$, and any element $\tau \in H^2(X)^\Hdg_{R(X,K')}$, we define
$$\kappa_\tau := \Eu_\tau \star_\tau \in \mathrm{End}_{R(X,K)}\(H^*(X,K)_{R(X,K)}\),$$
where $K$ is a field extension $\QQ \subset K \subset \CC$ containing all exponentials
$e^{\int_{\beta}\tau_2}$ of the integral of the degree-$2$ constant part $\tau_2$ of $\tau$ over any effective curve class $\beta$.
\end{dfn}

\begin{exa}\label{divisor rescaling}
In Definition \ref{matrix}, consider the case where
$$\tau':=\tau + \frac{2 \pi \ci p}{q} u ~,~~\quad \textrm{with }u \in H^2(X,\ZZ)~,~~\frac{p}{q} \in \QQ,$$
such that $\tau \in H^2(X)^\Hdg_{R(X,\QQ)}$
has no constant terms in degree $2$.
By Remark \ref{change of variables}, the endomorphism $\kappa_{\tau'}$ can be defined by extending our scalars to include $q$-th roots of unity.
\end{exa}

\begin{pro}\label{preserves HS}
For any $i \in \ZZ$, the endomorphism $\kappa_{T_0 \alpha_0 + \dotsb + T_h \alpha_h}$ is homogeneous of degree $2$ and is a morphism of Hodge structures
$$H^i(X,\QQ)_{R(X,\QQ)} \to H^{i+2}(X,\QQ)_{R(X,\QQ)}(1).$$
In particular, it stabilizes the subspaces $H^*(X)^{\Hdg}_{R(X,\QQ)}$ and $H^{(k)}(X)_{R(X,\QQ)}$.
\end{pro}

\begin{proof}
Consider a non-zero element
$$A:=\langle T_{i_0}\alpha_{i_0}, \phi_j, \phi_k, T_{i_1} \alpha_{i_1}, \dotsc, T_{i_n} \alpha_{i_n} \rangle_{0,3+n,\beta}^X \cfrac{Q^\beta}{n!} \phi^k.$$
By the dimension equation, we have
$$\sum_{k=0}^n \deg\(\alpha_{i_k}\) + \deg\(\phi_j\) + \deg\(\phi_k\) = 2\(\dim_\CC(X)-3+3+n+\int_{\beta}c_1(X)\),$$
which implies
\begin{eqnarray*}
\deg\(A\) & = & \sum_{k=0}^n \deg\( T_{i_k} \) + \deg \(Q^\beta\) + \deg\(\phi^k\) \\
& = & \sum_{k=0}^n \(2-\deg\( \alpha_{i_k} \)\) + 2 \int_{\beta} c_1(X) + 2 \dim_\CC(X) -  \deg\(\phi_k\) \\
& = & \deg\(\phi_j\)+2.
\end{eqnarray*}
Since the element $\kappa_\tau(\phi_j)$ is a formal power series with components of this form, $\kappa_\tau$ is homogeneous of degree $2$.
Furthermore, every constituent in the definition of $\kappa_\tau$ respects Hodge structures; specifically:
\begin{itemize}
\item the virtual cycle defining the correlators is an algebraic cycle of pure dimension,
\item the class $\tau$ and the Euler vector field $\Eu_\tau$ are Hodge classes,
\item the perfect pairing defining the dual basis is compatible with Hodge structures.
\end{itemize}
Hence the linear map
$$u \mapsto \langle \alpha_{i_0},u,\phi_k,\alpha_{i_1},\dotsc,\alpha_{i_n}\rangle_{0,n+3,\beta}^X \phi^k$$
is a morphism of Hodge structures
$$H^i(X,\QQ) \to H^j(X,\QQ)\(\cfrac{i-j}{2}\)~, \quad j:=\deg\(\phi^k\),$$
where the difference $i-j$ is even by the dimension equation.
Therefore, the endomorphism
$$\kappa_\tau \colon H^i(X,\QQ)_{R(X,\QQ)} \to H^{i+2}(X,\QQ)_{R(X,\QQ)}(1)$$
is a morphism of Hodge structures.
In particular, it preserves the Hochschild degree and the subspace of Hodge classes.
\end{proof}

\subsection{Examples}
We compute the endomorphism $\kappa_\tau$ in two important cases: when $X$ is a surface with a nef (numerically effective) canonical class and when $X$ is a cubic fourfold.

\begin{exa}[Surfaces with nef canonical class]\label{Knef}
Let $S$ be a smooth projective surface with $K_S \geq 0$, i.e.~it is numerically effective.
We choose a graded basis $\mathcal{H}=\(\alpha_0, \dotsc, \alpha_h\)$ of $H^*(S)^\Hdg$ such that
\begin{eqnarray*}
\alpha_0 & = & \un_S \in H^0(S,\CC), \\
\alpha_h & = & [\pt] \in H^4(S,\CC).
\end{eqnarray*}
We fix a graded basis $\mathcal{B} = \(\phi_0, \dotsc, \phi_m\)$ of $H^*(S,\QQ)$ with $\phi_0=\un_S$, $\phi_m=[\pt]$, and
$$\phi_i \in H^k(S,\CC) \textrm{ for $N_{k-1}<i \leq N_k$},$$
where $N_0=0$, $N_1 \leq N_2 \leq N_3=m-1$, and $N_4=m$.
Furthermore, we impose the condition
$$\phi_{N_2}=c_1(S), \textrm{ whenever }c_1(S) \neq 0.$$

Fix an element $\tau := f_0 \alpha_0 + \dotsb + f_h \alpha_h \in H^2(X)^\Hdg_{R(X,\QQ)}$.
By the string and divisor equations, we find
\begin{eqnarray*}
\kappa_\tau(\phi_j) & = & f_0 \phi_j \\
& & + \sum_{\beta} \exp \(\int_{\beta}\tau_2\) \(\sum_{k,n} \langle c_1(S),\phi_j,\phi_k,[\pt]^{\otimes n}\rangle_{0,n+3,\beta}^S \cfrac{f_h^n\phi^k}{n!}\) Q^\beta \\
& & - \sum_{\beta} \exp \(\int_{\beta}\tau_2\) \(\sum_{k,n} \langle [\pt],\phi_j,\phi_k,[\pt]^{\otimes n}\rangle_{0,n+3,\beta}^S \cfrac{f_h^{1+n}\phi^k}{n!}\)Q^\beta,
\end{eqnarray*}
where $\tau_2 := f_1 \alpha_1 + \dotsb + f_{h-1} \alpha_{h-1}$ corresponds to the degree-two Hodge classes.
The virtual dimension formula gives
$$\vdim = -1+n+3-K_S \cdot \beta,$$
so that
\begin{eqnarray*}
	\langle c_1(S),\phi_j,\phi_k,[\pt]^{\otimes n}\rangle_{0,n+3,\beta}^S \neq 0 & \implies & -1+n+3-K_S \cdot \beta = 1 + \cfrac{\deg(\phi_j)+\deg(\phi_k)}{2}+2n \\
	& \implies & \cfrac{\deg(\phi_j)+\deg(\phi_k)}{2}+n+K_S \cdot \beta=1, \\
	\langle [\pt],\phi_j,\phi_k,[\pt]^{\otimes n}\rangle_{0,n+3,\beta}^S \neq 0 & \implies & \cfrac{\deg(\phi_j)+\deg(\phi_k)}{2}+n+K_S \cdot \beta=0,
\end{eqnarray*}
resulting in the following possibilities.
For the first correlator, we must have one of the following:
\begin{itemize}
	\item $n=1$, $K_S \cdot \beta=0$, $\phi_j=\phi_k=\un_S$,
	\item $n=0$, $K_S \cdot \beta = 1$, $\phi_j=\phi_k=\un_S$,
	\item $n=0$, $K_S \cdot \beta=0$, $(\phi_j=\un_S,\phi_k \in H^2(S))$ or $(\phi_j \in H^2(S),\phi_k=\un_S)$,
	\item $n=0$, $K_S \cdot \beta=0$, $\phi_j$ and $\phi_k \in H^1(S)$.
\end{itemize}
By the string equation, the first and second cases result in vanishing correlators, and the condition $\beta=0$ must hold in the third case.
Specifically, this yields the contributions
$$\sum_k \langle c_1(S),\phi_j=\un_S,\phi_k\rangle_{0,3,0}^S \phi^k = c_1(S) \quad \textrm{and} \quad \langle c_1(S),\phi_j,\un_S\rangle_{0,3,0}^S [\pt]= \(\int_S c_1(S) \cup \phi_j\) [\pt].$$

For the second correlator, we must have
\begin{itemize}
\item $n=0$, $K_S \cdot \beta = 0$, $\phi_j=\phi_k=\un_S$,
\end{itemize}
and so $\beta=0$ by the string equation, yielding the contribution
$$\langle [\pt],\phi_j=\un_S,\un_S\rangle_{0,3,0}^S f_h [\pt] = f_h [\pt].$$
To sum up, we find
$$\(\kappa_\tau-f_0\)(\phi_j) \left\lbrace \begin{array}{lll}
= & c_1(S) - f_h [\pt] & \textrm{if $\phi_j=\un_S$}, \\
\in & H^3(S,\CC) & \textrm{if $\phi_j \in H^1(S,\CC)$}, \\
= & \(\int_S c_1(S) \cup \phi_j\) [\pt] & \textrm{if $\phi_j \in H^2(S,\CC)$}, \\
= & 0 & \textrm{if $\phi_j \in H^3(S,\CC)$ or if $\phi_j = [\pt]$},
\end{array}
\right. $$
that is, the matrix $\left[ \kappa_\tau - f_0 \right]^{\mathcal{B}}_{\mathcal{B}}$ is of the form
$$
\begin{pmatrix}
0&0&0&0&0\\
0&0&0&0&0\\
E&0&0&0&0\\
0&U&0&0&0\\
-f_h& 0&F &0&0
\end{pmatrix},
$$
where the blocks correspond to cohomology degrees, $U$ is the intersection pairing matrix on odd cohomology,
$$E= \begin{pmatrix}
0\\
\vdots \\
0 \\
\delta_{c_1(S)\neq 0}\\
\end{pmatrix} ~,\quad \textrm{and} \quad
F= \begin{pmatrix}
\int_S c_1(S) \cup \phi_{N_1+1} & \dotsc & \int_S c_1(S) \cup \phi_{N_2} \\
\end{pmatrix}.
$$
In particular, we observe that this matrix is nilpotent.
\end{exa}

\begin{exa}[Givental]\label{Givental}
Let $X$ be a smooth cubic hypersurface in $\PP^5$.
Its Novikov ring consists of polynomials in a single variable $Q$, as $H_2(X,\ZZ)$ is generated by a single class and $Q$ is assigned a non-zero degree.
We choose
$$\tau:=0 \in H^2(X)^\Hdg_{R(X,\QQ)}.$$

The cohomology of $X$ splits as the direct sum of the ambient part (pulled back from $\PP^5$) and the primitive part.
By the divisor equation, the endomorphism $\kappa_\tau$ vanishes on the primitive part.

Following Givental's computation and \cite[Proof of Theorem 6.8]{KKPY}, the endomorphism $\kappa_{\tau=0}$ on the ambient part equals
$$M = 3 \begin{pmatrix}
0 & 0 & 6Q & 0 & 0\\
1 & 0 & 0 & 15Q & 0\\
0 & 1 & 0 & 0 & 6Q\\
0 & 0 & 1 & 0 & 0\\
0 & 0 & 0 & 1 & 0\\
\end{pmatrix},
$$
expressed in the basis $\(1,H,H^2,H^3,H^4\)$, with $H$ the hyperplane class.
Written in the basis $\(1,H,H^2,H^3-21Q,H^4-6QH\)$, it becomes
$$M' := 3 \begin{pmatrix}
0 & 0 & 27Q & 0 & 0\\
1 & 0 & 0 & 0 & 0\\
0 & 1 & 0 & 0 & 0\\
0 & 0 & 1 & 0 & 0\\
0 & 0 & 0 & 1 & 0\\
\end{pmatrix}.
$$
We deduce that $0$ is an eigenvalue with a $1$-dimensional eigenspace, and its generalized eigenspace has dimension $2$, given by
$$\ker\(\kappa_{\tau=0}^m\) = \mathrm{Vect}\(H^3-21Q,H^4-6QH\)~,~~\textrm{for }m \geq 2.$$
Note that if we extend scalars to $\CC[[Q^{1/3}]]$, we obtain three other eigenvalues $9Q^{1/3}, 9e^{2\ci\pi/3}Q^{1/3},9e^{-2\ci\pi/3}Q^{1/3}$.
\end{exa}

\section{Birational transformation}
In this section, we consider a number field $K$ and the (rational) Levi-Civita field
$$F:=\QQ((a^\QQ)),$$
whose elements are series in the formal variable $a$ of the form
$$\sum_{q \in \QQ} x_q a^q~,~~x_q \in \QQ$$
where for any rational number $r$ there are only finitely many non-zero coefficients $x_q$ with $q \leq r$. The absolute value is defined by
$$|\sum_{q \geq p} x_q a^q| := 2^{-p}~,~~x_p \neq 0.$$
Moreover, we declare that every element of $F$ has degree zero.

We introduce an additional formal variable $b$, which we declare to have degree $1$ and absolute value $|b|=1$. We then define the graded $\QQ$-algebra 
$$S^*:=F[b^{\pm 1}] = \bigoplus_{k \in \ZZ} S^k ~,\quad \textrm{with } S^k = F \cdot b^k.$$
We write $S^*_K := S^* \otimes_\QQ K$ for the corresponding graded $K$-algebra.

\begin{lem}\label{conv crit}
We have the equivalence
$$\sum_{n \in \NN} x_n \in F_K \iff x_n \to 0,$$
for any sequence $\(x_n\)_{n \in \NN} \in \(F_K\)^\NN$.
\end{lem}

\begin{proof}
If $x_n \to 0$, we can write
$$\forall n \in \NN ~,~~x_n = \sum_{q \geq w_n} y_{n,q} a^q~,~~y_{n,q} \in K,$$
where the sequence $\(w_n\)_{n \in \NN} \in \QQ^\NN$ tends to infinity.
This means that for any $Q \in \QQ$, there exists $N \in \NN$ such that for all $n > N$, we have $w_n > Q$.
It follows that
$$\sum_{\substack{n \in \NN \\ w_n \leq q \leq Q}} y_{n,q} a^q = \sum_{q \leq Q} z_q a^q~,~~\textrm{where }z_q := \sum_{\substack{0 \leq n \leq N \\ w_n \leq q}} y_{n,q}.$$
By definition of $F_K$, the sum defining $z_q$ is finite and is non-empty for only finitely many values of $q$.
Hence, the element $\sum_{n \in \NN} x_n$ is well-defined in $F_K$.
The converse implication holds in any metric space.
\end{proof}

\subsection{Evaluation maps}
Let $f \colon X \to Z$ be a morphism of varieties.
We denote the Novikov variable of $Z$ by $Q'$ and an ample class in $Z$ by  $\omega$.
We denote by $T_0, \dotsc, T_h$ the formal variables associated with a graded basis $\mathcal{H}$ of $H^*(X)^\Hdg$.
Introducing an additional formal variable $\mathfrak{q}'$ of degree in $\left\lbrace 1,2\right\rbrace$ (to be specified later), we define the graded $\QQ$-algebra $R^*_f$ as
$$R^*_f := \QQ[{\mathfrak{q}'}^{\pm 1}][[Q',T_0, \dotsc, T_h]].$$
Note that we have a natural inclusion $R^*(X,K) \subset R^*_{\id_X,K}$ when the morphism is the identity $\id_X \colon X \to X$.

\begin{rem*}
The morphism $f \colon X \to Z$ induces a morphism of Novikov rings $\QQ[[Q]] \to \QQ[[Q']]$. It is important to note that this morphism is not generally homogeneous of degree zero, as will be seen in the case of the blow-up map $\widetilde{X} \to X$ in the next section.
To correct it, we make the assumption below.
\end{rem*}

\begin{assumption}\label{assump}
Throughout this paper, we assume that the morphism $f \colon X \to Z$ is equipped with a degree-zero morphism $J \colon \QQ[[Q]] \to \QQ[{\mathfrak{q}'}^{\pm 1}][[Q']]$ of the form
$$Q^\beta \mapsto {Q'}^{j(\beta)} {\mathfrak{q}'}^{j_\beta}.$$
In particular, this induces a degree-zero homogeneous morphism
$$R^*(X,K) \to R^*_{f,K}$$
between the corresponding graded $K$-algebras.
\end{assumption}

\begin{rem}
If $J(Q^\beta)=1$ for some $\beta \neq 0$, then $Q^\beta$ is homogeneous of degree zero.
However, in this case, the morphism $J$ is not well-defined on the formal power series $\sum_{N \geq 0} Q^{N \beta} \in \QQ[[Q]]$, as the image would involve an infinite sum of constant terms, leading to a contradiction.
Consequently, if $\beta \neq 0$ and $j(\beta)=0$, it must follow that $j_\beta \neq 0$. 
\end{rem}

\begin{exa}
We have three main examples for the morphism $f \colon X \to Z$, namely:
\begin{itemize}
\item the identity map $\id_X \colon X \to X$,
\item the blow-up map $\mathrm{pr} \colon \mathrm{Bl}_{X'}(X) \to X$ of a smooth center $X' \subset X$,
\item the inclusion map of the smooth center $X' \subset X$.
\end{itemize}
We will see below in Section \ref{sect Bl} that they all satisfy Assumption \ref{assump}.
\end{exa}

\begin{dfn}[Evaluation function]\label{eval}\label{rescale}
Let
$$\ev \colon \left\lbrace \mathfrak{q}', {Q'}^\beta,T_0, \dotsc, T_h\right\rbrace  \to S^*_K$$
be a degree-zero homogeneous function that is multiplicative in $\beta \in \mathrm{NE}_\NN(Z)$.	
We say that $\ev$ is a \textit{normalized $K$-evaluation function} for the morphism $X \to Z$ if
\begin{itemize}
\item we have $\ev(\mathfrak{q}')=b^{\deg(\mathfrak{q}')}$,
\item there exists $0<\epsilon<1$ such that for any non-zero curve class $\beta \in \mathrm{NE}_\NN(Z)$, we have
$$0 \leq |\ev({Q'}^\beta)|<\epsilon^{\int_{\beta} \omega} < 1,$$
\item for each $0 \leq k \leq h$, we have
$$0 \leq |\ev(T_k)| <1.$$
\end{itemize}

For any $\lambda \in F_K$ and any function $\ev$, we define the function $\lambda\ev$ by
$$\forall ~x \in \left\lbrace \mathfrak{q}', {Q'}^\beta,T_0, \dotsc, T_h\right\rbrace~,~~\lambda\ev(x) := \lambda^{\deg(x)}\ev(x).$$
We say that $\ev$ is a \textit{$K$-evaluation function} for the morphism $X \to Z$ if there exists $\lambda \in F_K$ such that $\lambda\ev$ is a normalized $K$-evaluation function.
\end{dfn}

\begin{pro}[{\cite[Lemma 4.6]{KKPY}}]\label{conv}
Every $K$-evaluation function for the morphism $f \colon X \to Z$ can be extended uniquely and continuously to a degree-zero homogeneous morphism of graded $K$-algebras
$$R^*_{f,K} \to S^*_K,$$
that we call a $K$-evaluation map.
Conversely, every degree-zero homogeneous morphism $\ev \colon R^*_{f,K} \to S^*_K$ of graded $K$-algebras satisfies:
$$\forall ~x \in \left\lbrace \mathfrak{q}', {Q'}^\beta,T_0, \dotsc, T_h\right\rbrace~,~~ 0 \leq |\ev(x)| < |\ev(\mathfrak{q'})|^{\frac{\deg(x)}{\deg(\mathfrak{q'})}}.$$
\end{pro}

\begin{proof}
By homogeneity, the statement holds if and only if it holds for all \textit{normalized} $K$-evaluation functions, so that without any loss of generality, we may assume $|\ev(\mathfrak{q}')|=1$ and apply the criterium of Lemma \ref{conv crit} to prove that the element
$$\sum_{N \in \ZZ} \sum_{\substack{k \in \ZZ \\ \underline{n} \in \NN^{1+h} \\ \beta \in \mathrm{NE}_\NN(Z) \\ \int_{\beta}\omega + |\underline{n}|+k=N}} c_{\beta, \underline{n}, k} \ev\({Q'}^\beta T_0^{n_0} \dotsc T_h^{n_h} {\mathfrak{q}'}^k\)$$
is well-defined in $F_K$, where we adopt the notation $\underline{n}:=\(n_0, \dotsc, n_{h+l}\) \in \NN^{1+h+l}$ and $|\underline{n}|:=n_0+\dotsb+n_{h+l}$.
We examine the limit as $N \to \pm \infty$ of
\begin{eqnarray*}
|A_N| & := & |\sum_{\substack{k \in \ZZ \\ \beta \in \mathrm{NE}_\NN(Z) \\ \int_{\beta}\omega + |\underline{n}|+k=N}} c_{\beta, \underline{n}, k}~ \ev\({Q'}^\beta T_0^{n_0} \dotsc T_h^{n_h} {\mathfrak{q}'}^k\) | \\
& \leq & \max \left\lbrace |\ev\({Q'}^\beta\)| \cdot |\ev\(T_0\)|^{n_0} \dotsc |\ev\(T_h\)|^{n_h} \cdot  |\ev\({\mathfrak{q}'}^k\)| \right\rbrace \\
& \leq & \max \left\lbrace \epsilon^{\int_\beta \omega+|\underline{n}|} \right\rbrace,
\end{eqnarray*}
where the maximum is taken over the finite set of all
$$k \in \ZZ ~,~~ \underline{n} \in \NN^{1+h}~,~~ \beta \in \mathrm{NE}_\NN(Z),$$
such that $\int_{\beta}\omega + |\underline{n}|+k=N$ and $c_{\beta, \underline{n},k} \neq 0$.
By the definition of a graded algebra, there exists $D>0$ such that
$$c_{\beta, \underline{n}, k} \neq 0 \implies -D \leq k \deg\(\mathfrak{q}'\) + 2\int_{\beta}c_1(Z)+ \sum_{i=0}^h n_i \deg\(T_i\) \leq D.$$
This implies the following bounds:
\begin{eqnarray*}
N \leq & \frac{D}{\deg\(\mathfrak{q}'\)}+ \(1-\frac{2\min\(-1,\frac{\int_{\beta_1} c_1(Z)}{\int_{\beta_1}\omega}, \dotsc, \frac{\int_{\beta_l} c_1(Z)}{\int_{\beta_l}\omega}\)}{\deg\(\mathfrak{q}'\)}\) \(\int_\beta \omega+|\underline{n}|\), \\
-N \leq & \frac{D}{\deg\(\mathfrak{q}'\)}+ \(\frac{2\max\(1, \frac{\int_{\beta_1} c_1(Z)}{\int_{\beta_1}\omega}, \dotsc, \frac{\int_{\beta_l} c_1(Z)}{\int_{\beta_l}\omega}\)}{\deg\(\mathfrak{q}'\)}-1\) \(\int_\beta \omega+|\underline{n}|\), \\
\end{eqnarray*}
where $\(\beta_1, \dotsc, \beta_l\)$ is the dual basis to a basis of ample classes $\(\omega_1,\dotsc, \omega_l\)$ for the torsion-free part of the N\'eron--Severi group of $Z$, so that every effective curve is a linear combination of $\beta_1, \dotsc, \beta_l$ with non-negative integer coefficients.
Consequently, we obtain $A_N \to 0$ as $N \to \pm \infty$.

Conversely, let $\ev \colon R^*_{f,K} \to S^*_K$ be a degree-zero homogeneous morphism of graded $K$-algebras.
In particular, since $\ev(\mathfrak{q}') \neq 0$, we may assume that $\ev(\mathfrak{q}')=b^{\deg(\mathfrak{q}')}$ after potentially rescaling $\ev$ by the factor $\lambda:=\ev(\mathfrak{q}')^{-\frac{1}{\deg(\mathfrak{q}')}}b$.

Suppose there exists $x \in \left\lbrace {Q'}^\beta,T_0, \dotsc, T_h\right\rbrace$ such that $|\ev(x)|\geq 1$.
Then the element $y:=x {\mathfrak{q}'}^{-\frac{\deg(x)}{\deg(\mathfrak{q}')}}$ is homogeneous of degree $0$.
It follows that the formal power series $\sum_{N \geq 0} y^N$ is well-defined in $R^*_Z$, whereas its image $\sum_{N \geq 0} \ev(y)^N$ is not well-defined in $S^*_K$ by Lemma \ref{conv crit}. This yields a contradiction.
\end{proof}

\begin{pro}[Preservation of Hodge structures]\label{span HS}
Let $\ev \colon R^*_{f,K} \to S^*_K$ be a $K$-evaluation function, and let $\(s_1,\dotsc, s_N\)$ be a $\QQ$-basis of $K$, where $N:=\dim_\QQ K$.
We set $\tau := T_0 \alpha_0 + \dotsb + T_h \alpha_h$.

Then the endomorphism $\ev\(\kappa_\tau\)$ is well-defined and lies in the $K$-linear span of morphisms of Hodge structures (see Definition \ref{Kspan}).
Specifically, for any $i \in \ZZ$, there exist morphisms of free $S^\mathrm{even}$-modules and of Hodge structures
$$\kappa_1, \dotsc, \kappa_N \colon H^i(X,\QQ)_S \to H^{i+2}(X,\QQ)_S(1)$$
such that $\ev\(\kappa_\tau\) = s_1 \kappa_1+ \dotsb + s_N \kappa_N$.
In particular, $\ev\(\kappa_\tau\)$ stabilizes the subspaces $H^*(X)^{\Hdg}_S$ and $H^{(k)}(X)_{S_\CC}$.
\end{pro}

\begin{proof}
By Proposition \ref{conv}, we have a well-defined degree-zero homogeneous morphism $\ev \colon R^*_{f,K} \to S^*_K$ of graded $K$-algebras and thus a morphism
$$\ev \colon R^*(X,K) \to R^*_{f,K} \to S^*_K.$$
Recall that $\kappa_\tau$ is an endomorphism of $H^*(X,\QQ)_{R(X,\QQ)}$.
We define an endomorphism $\ev(\kappa_\tau) \in \mathrm{End}_{S^*_K}(H^*(X,K)_{S_K})$ as follows: choose a basis $\mathcal{B}$ of $H^*(X,K)$ and define $[\ev(\kappa_\tau)]^\mathcal{B}_\mathcal{B}:=\ev\([\kappa_\tau]^\mathcal{B}_\mathcal{B}\)$.
Note that it is independant of the choice of the basis $\mathcal{B}$ of $H^*(X,K)$, because we have $\ev(P)=P$ for any base change matrix $P \in \mathrm{Mat}(K)$.

Every element $\zeta \in F_K$ decomposes uniquely as a sum
$$\zeta = s_1 \zeta_1 + \dotsb + s_N \zeta_N,$$
where $\zeta_1, \dotsc, \zeta_N \in F$.
Let $\ev\(\tau\) = g_0 \alpha_0 + \dotsb + g_h \alpha_h$, where $g_k \in S^{\deg\(T_k\)}_K$ for each $0 \leq k \leq h$.
We compute
\begin{eqnarray*}
\ev\(\kappa_\tau\)(\phi_i) & = & \sum_{\substack{0 \leq k \leq m, n \geq 0 \\ \beta \in \mathrm{NE}_\NN(X)}} \langle c_1(X) + \sum_{l=0}^h \(1-\cfrac{\deg \alpha_l}{2}\) g_l \alpha_l,\phi_i,\phi_k,\(g_0 \alpha_0 + \dotsb + g_h \alpha_h\)^{\otimes n}\rangle_{0,3+n,\beta}^X \cfrac{\ev\(Q^\beta\)}{n!} \phi^k \\
& = & \sum_{\substack{0 \leq k \leq m \\ n_0, \dotsc, n_h \geq 0 \\ n_0 + \dotsb + n_h = n \\ \beta \in \mathrm{NE}_\NN(X)}} g_0^{n_0} \dotsm g_h^{n_h} \cfrac{\langle c_1(X),\phi_i,\phi_k, \alpha_0^{\otimes n_0},\dotsc,\alpha_h^{\otimes n_h}\rangle_{0,3+n,\beta}^X}{n_0! \dotsm n_h!} \ev\(Q^\beta\) \phi^k \\
&  & + \sum_{\substack{0 \leq k \leq m \\ 0 \leq l \leq h \\ n_0, \dotsc, n_h \geq 0 \\ n_0 + \dotsb + n_h = n \\ \beta \in \mathrm{NE}_\NN(X)}} \(1-\cfrac{\deg \alpha_l}{2}\) g_l g_0^{n_0} \dotsm g_h^{n_h} \cfrac{\langle \alpha_l,\phi_i,\phi_k, \alpha_0^{\otimes n_0},\dotsc,\alpha_h^{\otimes n_h}\rangle_{0,3+n,\beta}^X}{n_0! \dotsm n_h!} \ev\(Q^\beta\) \phi^k \\
& = & \sum_{j=1}^N s_j \left(  \sum_{\substack{0 \leq k \leq m \\ n_0, \dotsc, n_h \geq 0 \\ n_0 + \dotsb + n_h = n \\ \beta \in \mathrm{NE}_\NN(X)}} h^{(j)}_{n_0,\dotsc, n_h, \beta} \cfrac{\langle c_1(X),\phi_i,\phi_k, \alpha_0^{\otimes n_0},\dotsc,\alpha_h^{\otimes n_h}\rangle_{0,3+n,\beta}^X}{n_0! \dotsm n_h!} \phi^k \right.  \\
&  & \left. + \sum_{\substack{0 \leq k \leq m \\ 0 \leq l \leq h \\ n_0, \dotsc, n_h \geq 0 \\ n_0 + \dotsb + n_h = n \\ \beta \in \mathrm{NE}_\NN(X)}} \(1-\cfrac{\deg \alpha_l}{2}\) h^{(j)}_{l;n_0,\dotsc, n_h, \beta} \cfrac{\langle \alpha_l,\phi_i,\phi_k, \alpha_0^{\otimes n_0},\dotsc,\alpha_h^{\otimes n_h}\rangle_{0,3+n,\beta}^X}{n_0! \dotsm n_h!} \phi^k \right),
\end{eqnarray*}
with some elements $h^{(j)}_{n_0,\dotsc, n_h, \beta},h^{(j)}_{l;n_0,\dotsc, n_h, \beta} \in S^*$
and where
$$\phi_i \mapsto \langle x,\phi_i,\phi_k, \alpha_0^{\otimes n_0},\dotsc,\alpha_h^{\otimes n_h}\rangle_{0,3+n,\beta}^X \phi^k~,~~x \in \left\lbrace c_1(X),\alpha_l\right\rbrace,$$
is a morphism of Hodge structures from $H^*(X,\QQ)$ to $H^{*+2}(X,\QQ)\(1\)$.
Hence, denoting by $\kappa_j$ the morphism within the parentheses, we see that it is a morphism of free $S^\mathrm{even}$-modules and of Hodge structures from $H^*(X,\QQ)_S$ to $H^{*+2}(X,\QQ)_S\(1\)$.
\end{proof}

In the next two Propositions, we compare $K$-evaluation maps for a morphism $f \colon X \to Z$ and for the identity $\id_X \colon X \to X$.

\begin{pro}\label{incr2}
	Let $f \colon X \to Z$ be a morphism satisfying Assumption \ref{assump} and $\ev \colon R^*_{f,K} \to S^*_K$ be a $K$-evaluation map.
	Here, we distinguish the formal variables $\mathfrak{q}'_X$ and $\mathfrak{q}'_Z$ associated with $R^*_{\id_X,K}$ and $R^*_{f,K}$, respectively.
	
	Then, there exists a $K$-evaluation map $\ev' \colon R^*_{\id_X,K} \to S^*_K$ making the following diagram commutative
	$$
	\begin{tikzcd}[row sep=large, column sep=large]
	& R^*_{f, K} \arrow[rd, "\ev"] & \\
	R^*(X, K) \arrow[ru] \arrow[rd, hook] & \circlearrowright & S^*_K \\
	& R^*_{\id_X, K} \arrow[ru, dashed, "\exists\ev' "] & 
	\end{tikzcd}
	$$
	if and only if for any non-zero curve classes $\beta_1, \beta_2 \in \mathrm{NE}_\NN(X)$, we have
	$$\(J(Q^{\beta_1}) = {\mathfrak{q}'}_Z^{j_1} \quad \textrm{and} \quad J(Q^{\beta_2}) = {\mathfrak{q}'}_Z^{j_2}\) \implies j_1 j_2 >0.$$
\end{pro}

\begin{rem*}
As it will be shown in the proof of Lemma \ref{ext}, both the blow-up map and the inclusion of a smooth center satisfy the condition on the morphism $J$ written at the end of Proposition \ref{incr2}.
\end{rem*}

\begin{proof}
	Let $\lambda \in F_K$ be such that $\ev_1:=\lambda\ev$ is normalized.
	By composition, we obtain a morphism $\ev_2 \colon R^*(X,K) \to S^*_K$, which we aim to extend to $R^*_{\id_X,K}$.
	We have $|\ev_2(T_k)| = |\ev_1(T_k)| <1$
	for each $0 \leq k \leq h$, and $|\ev_2(Q^\beta)| = |\ev_1({Q'}^{j(\beta)})| < 1$
	for any curve class $\beta \in \mathrm{NE}_\NN(X)$ such that $j(\beta) \neq 0$.
	
	Suppose there exist $\beta_1$ and $\beta_2$ such that
	$$J(Q^{\beta_1}) = {\mathfrak{q}'}_Z^{j_1} \quad \textrm{and} \quad J(Q^{\beta_2}) = {\mathfrak{q}'}_Z^{j_2}$$
	with $j_1<0<j_2$.
	Then $|\ev_2(Q^{\beta_1})|=|\ev_2(Q^{\beta_2})|=1$.
	Consequently, for any $\lambda \in F_K$ with $|\lambda| \neq 1$, one of the values $|\lambda\ev_2(Q^{\beta_1})|$ or $|\lambda\ev_2(Q^{\beta_2})|$ will strictly exceed $1$, making it impossible to obtain a normalized $K$-evaluation function from $\ev_2$.
	
	Now assume that for any $\beta_1$ and $\beta_2$ such that $J(Q^{\beta_1}) = {\mathfrak{q}'}_Z^{j_1}$ and $J(Q^{\beta_2}) = {\mathfrak{q}'}_Z^{j_2}$, the integers $j_1$ and $j_2$ have the same sign.
	We treat the case $j_1, j_2 >0$; the other case is analogous.
	
	We may choose $\eta>0$ sufficiently small to define a $K$-evaluation map $\ev'_1$ for $X$ by setting $\ev'_1(\mathfrak{q}'_X):=b^{\deg(\mathfrak{q}'_X)}$ and $\ev'_1:=a^\eta \ev'$ on $\left\lbrace Q, T_0, \dotsc, T_h \right\rbrace$.
	Indeed, given any number $0<\epsilon<1$, we choose $\eta>0$ such that
	$$|\ev_2(Q^{\beta})|<\epsilon^{\int_{\beta}\omega} < 2^{-\eta |\deg(Q^{\beta})|}$$
	for any curve class $\beta \in \mathrm{NE}_\NN(X)$ of negative degree. This is satisfied if
	$$0<\eta \frac{|\deg(Q^{\beta})|}{\int_{\beta}\omega }< \eta \max\(1,\frac{|\int_{\beta_1} c_1(X)|}{\int_{\beta_1}\omega}, \dotsc, \frac{|\int_{\beta_l} c_1(X)|}{\int_{\beta_l}\omega}\) <\log_2(\epsilon^{-1}),$$
	where $(\beta_1, \dotsc, \beta_l)$ is the basis defined in the proof of Proposition \ref{conv}.
	Finally, we set
	$$\ev':=a^{-\eta}\lambda^{-1}\ev'_1 \colon R^*_{\id_X,K} \to S^*_K$$
	which provides the desired factorization.
\end{proof}

It is worth noting that the converse Proposition requires extra assumptions.

\begin{pro}\label{incr3}
	Let $f \colon X \to Z$ be a morphism satisfying Assumption \ref{assump} and  $\ev \colon R^*_{\id_X,K} \to S^*_K$ be a $K$-evaluation map.
	Here, we distinguish the variables $\mathfrak{q}'_X, Q_X$ and $\mathfrak{q}'_Z, Q_Z$ associated with $R^*_{\id_X,K}$ and $R^*_{f,K}$, respectively.
	
	We assume
	\begin{itemize}
	\item there exists an effective curve class $\beta_0 \in \mathrm{NE}_\NN(X)$ such that $j_{\beta_0} > 0$ and $j(\beta_0)=0$,
	\item for any effective curve class $\beta \in \mathrm{NE}_\NN(X)$, we have the divisibility $j_{\beta_0} | j_\beta$,
	\item the map $j \colon \mathrm{NE}_\NN(X) \to \mathrm{NE}_\NN(Z)$ is surjective,
	\item the morphism $J \colon \QQ[[Q_X]] \to \QQ[{\mathfrak{q}'_Z}^{\pm 1}][[Q_Z]]$ is injective.
	\end{itemize}
	
	Then, there exists a $K$-evaluation map $\ev' \colon R^*_{f,K} \to S^*_K$ making the following diagram commutative
	$$
	\begin{tikzcd}[row sep=large, column sep=large]
	& R^*_{f, K} \arrow[rd, dashed, "\exists\ev'"] & \\
	R^*(X, K) \arrow[ru] \arrow[rd, hook] & \circlearrowright & S^*_K \\
	& R^*_{\id_X, K} \arrow[ru, "\ev "] & 
	\end{tikzcd}
	$$
	if and only if there exists $0 < \epsilon < 1$ such that
	$$|\ev(Q_X^\beta)|< |\ev(Q_X^{\beta_0})|^{\frac{\deg(Q_X^\beta)}{\deg(Q_X^{\beta_0})}} \epsilon^{\int_{j(\beta)} \omega},$$
	for any non-zero effective curve class $\beta \in \mathrm{NE}_\NN(X)$ ($\omega$ here is an ample class in $Z$).
\end{pro}

\begin{rem*}
As it will be shown in the proof of Lemma \ref{ext}, the blow-up map satisfies the assumptions of Proposition \ref{incr3}.
However, the inclusion of a smooth center does not, which will force us later to make a technical distinction between properties $\diamondsuit$ and $\diamondsuit^{\mathrm{strong}}$, see Definition \ref{key}.
\end{rem*}

\begin{proof}
First, we must define $\ev'(T_i)=\ev(T_i)$ for each $0 \leq i \leq h$.
We must also satisfy the condition
$$\ev'(J(Q_X^\beta)) = \ev(Q_X^\beta)$$
for any effective curve class $\beta \in \mathrm{NE}_\NN(X)$.
Since we have $\beta_0 \in \mathrm{NE}_\NN(X)$ such that $j_{\beta_0} > 0$ and $j(\beta_0)=0$, then we must also define
$$\ev'(\mathfrak{q}'_Z) = \ev(Q_X^{\beta_0})^{\frac{1}{j_{\beta_0}}}.$$

Let $\beta' \in \mathrm{NE}_\NN(Z)$.
By surjectivity of the map $j \colon \mathrm{NE}_\NN(X) \to \mathrm{NE}_\NN(Z)$, we can find $\beta \in \mathrm{NE}_\NN(X)$ such that $j(\beta)=\beta'$, so that we must define
$$\ev'(Q_Z^{\beta'}) = \ev(Q_X^{\beta}) \ev(Q_X^{\beta_0})^{-\frac{j_\beta}{j_{\beta_0}}}.$$
It is independant of the choice of $\beta$.
Indeed, if we have $j(\beta_1)=j(\beta_2)=\beta'$, then we have
$$J(Q_X^{\beta_2})={\mathfrak{q}'_Z}^{j_{\beta_2}}Q_Z^{\beta'}={\mathfrak{q}'_Z}^{j_{\beta_2}-j_{\beta_1}}J(Q_X^{\beta_1})=J(Q_X^{\beta_1+\frac{j_{\beta_2}-j_{\beta_1}}{j_{\beta_0}}\beta_0}).$$
By injectivity of the morphism $J \colon \QQ[[Q_X]] \to \QQ[{\mathfrak{q}'_Z}^{\pm 1}][[Q_Z]]$, we deduce
$$Q_X^{\beta_2}=Q_X^{\beta_1+\frac{j_{\beta_2}-j_{\beta_1}}{j_{\beta_0}}\beta_0},$$
so that
$$\ev'(Q_Z^{\beta'}) = \ev(Q_X^{\beta_1}) \ev(Q_X^{\beta_0})^{-\frac{j_{\beta_1}}{j_{\beta_0}}} = \ev(Q_X^{\beta_2}) \ev(Q_X^{\beta_0})^{-\frac{j_{\beta_2}}{j_{\beta_0}}}.$$

At this point, we have defined $\ev'$ in terms of $\ev$ on the generators.
It is a $K$-evaluation function if and only if it satisfies inequalities from Definition \ref{eval}, which are equivalent to the inequality claimed in the statement.
\end{proof}

\subsection{Non-archimedean geometry}\label{nonarchi}
In this section, we recall basic properties of matrices defined over the non-archimedean field $F_K$ that we will use later, e.g.~in the proof of Proposition \ref{incr}.

We consider a matrix
$$A := \sum_{N \geq 0} A_N t^N \in \mathrm{Mat}_{n,n}(F_K[[t]])$$
such that the evaluation
$$A(\zeta) := \sum_{N \geq 0} A_N \zeta^N \in \mathrm{Mat}_{n,n}(F_K)$$
is well-defined for all $\zeta \in D(0,1)$.
Here, $D(\zeta_0,r) \subset F_K$ denotes the open disk of center $\zeta_0 \in F_K$ and radius $r>0$, defined by
$$\zeta \in D(\zeta_0,r) \iff |\zeta - \zeta_0|<r.$$
Furthermore, for each $0 \leq i \leq p$, we assume there exists a $\CC$-subspace $H_i \subset \mathrm{Mat}_{n,1}(\CC)$
such that the module $H_i \otimes_{\CC} F_\CC[[t]]$ is stable under multiplication by $A_\CC$.

\begin{lem}\label{rk}
	For any $0 < r < 1$, the determinant function $\det \colon D(0,r) \to F_K$ is either identically zero or has only finitely many zeros.
	Consequently, the rank function $\rk \colon D(0,r) \to \NN$, given by $\zeta \mapsto \rk(A(\zeta))$, is constant except at finitely many points. That is, there exists a finite set $\mathcal{S} \subset D(0,r)$ such that the restriction $\rk_{|D(0,r)-\mathcal{S}}$ is constant.
	Furthemore, the rank function is lower semi-continuous.
\end{lem}

\begin{proof}
	The determinant $\det(A)$ is a non-zero power series with coefficients in $F_K$ that converges on the open disk $D(0,1)$, and thus on the closed disk $\overline{D(0,r)}$.
	More generally, any non-zero power series with coefficients in a non-archimedean field that converges on a closed disk has only finitely many zeros within that disk.
	
	The result concerning the rank function follows from the characterization of the rank as the size of a largest invertible sub-matrix.
\end{proof}

Fix $0 < r < 1$, and for any $\zeta \in D(0,r)$, let $\Sp(A(\zeta))$ be the set of eigenvalues of $A$ in $F_{\overline{\QQ}}$.

\begin{lem}\label{eigen}
	There is a finite set $\mathcal{S} \subset D(0,r)$ such that the function
	$$\zeta \mapsto \mathrm{card}\(\Sp(A(\zeta))\)$$
	is constant on $D(0,r) - \mathcal{S}$ and lower semi-continuous on $D(0,r)$.
	Furthermore, for any $\zeta_0 \in D(0,r) - \mathcal{S}$ and any $\alpha_0 \in \Sp(A(\zeta_0))$, there exists a power series $\alpha \in F_{\overline{\QQ}}[[t]]$ that converges in some non-empty open disk $D(\zeta_0,r')$ and satisfies
	$$\alpha(\zeta_0)=\alpha_0 \quad \textrm{and} \quad \alpha(\zeta) \in \Sp(A(\zeta)),$$
	for all $\zeta \in D(\zeta_0,r')$.
\end{lem}

\begin{proof}
	Since $F_K[[t]]$ is a unique factorization domain, we can write the characteristic polynomial as
	$$\det\(A(t)-X\) = Q_1(t,X)^{k_1} \dotsm Q_m(t,X)^{k_m} \in F_K[[t]][X], $$
	where $Q_1, \dotsc, Q_m$ are distinct irreducible polynomials in $F_K[[t]][X]$.
	We set
	$$P(t,X) := Q_1(t,X) \dotsm Q_m(t,X),$$
	so that $\mathrm{Sp}(A(\zeta)) = \left\lbrace \lambda \in F_{\overline{\QQ}} ~|~~P(\zeta,\lambda)=0 \right\rbrace$.
	We claim there is a finite set $\mathcal{S} \subset D(0,r)$ such that for any $\zeta \in D(0,1)- \mathcal{S}$, the polynomials $Q_1(\zeta,X), \dotsc, Q_m(\zeta,X)$ split with distinct roots in $\overline{\QQ}$, all with algebraic multiplicity one.
	In particular, for any $\zeta \in D(0,r)- \mathcal{S}$ we have
	$$\mathrm{card}\(\Sp(A(\zeta))\) = \deg\(P(t,X)\).$$
	
	To prove the claim, for any $1 \leq i < j \leq m$, we consider the resultant $\mathrm{Res}_{i,j}$ of $Q_i$ and $Q_j$ with respect to $X$.
	Since the resultant is defined as the determinant of a matrix with coefficients in the integral domain $F_K[[t]]$ and since $Q_i$ and $Q_j$ are distinct irreducible polynomials, $\mathrm{Res}_{i,j}$ is a non-zero element of $F_K[[t]]$.
	Thus, it has only finitely many zeros in $D(0,r)$.
	
	The main property of the resultant is the following: for any $\zeta \in D(0,r)$, we have $\mathrm{Res}_{i,j}(\zeta)=0$ if and only if there exists $\lambda \in \overline{\QQ}$ such that $Q_i(\zeta,\lambda) = Q_j(\zeta,\lambda)=0$.
	Consequently, the set of such values $\zeta \in D(0,r)$ is finite and the polynomials $Q_i(\zeta,X)$ and $Q_j(\zeta,X)$ have no common roots when $\zeta$ is outside this set.
	
	Similarly, for each $1 \leq i \leq m$, the resultant $\mathrm{Res}_i$ of the polynomials $Q_i$ and $\frac{\partial Q_i}{\partial X}$ is a non-zero element of $F_K[[t]]$, because $Q_i$ is irreducible.
	It follows that $Q_i(\zeta, X)$ has a multiple root only for finitely many $\zeta \in D(0,r)$, because $\mathrm{Res}_i$ has only finitely many zeroes in $D(0,r)$, and for any $\zeta \in D(0,r)$, $\mathrm{Res}_i(\zeta)=0$ if and only if there exists a multiple root of the polynomial $Q_i(\zeta,X)$ in $\overline{\QQ}$.
	
	Next, we prove the second statement.
	Let $\zeta_0 \in D(0,r) - \mathcal{S}$ and $\alpha_0 \in \Sp(A(\zeta_0))$.
	Since $\alpha_0$ is a simple root of $P(\zeta_0,X)$, we have $P(\zeta_0,\alpha_0)=0$ and $\frac{\partial P}{\partial X}(\zeta_0,\alpha_0) \neq 0$.
	The implicit function theorem for formal power series provides a unique $\alpha \in F_{\overline{\QQ}}[[t-\zeta_0]]$ such that $P(t,\alpha(t)) = 0 \in F_{\overline{\QQ}}[[t]]$.
	
	To check the convergence, assume $\zeta_0 = \alpha_0=0$ without loss of generality.
	Write $P(t,X) = \sum_{p, q \geq 0} c_{p,q} t^pX^q$. The conditions $P(\zeta_0,\alpha_0)=0$ and $\frac{\partial P}{\partial X}(\zeta_0,\alpha_0) \neq 0$ imply $c_{0,0} = 0$ and $c_{0,1} \neq 0$.
	The coefficients $\alpha_n$ of $\alpha(t)$ are determined by the recurrence:
	$$\forall ~ N \geq 1 ~,~~ c_{0,1} \alpha_N + \sum_{\substack{p,q \geq 0 \\ (p,q) \neq (0,1) \\ n_1, \dotsc, n_q \geq 0 \\ p+n_1+\dotsb+n_q = N}} c_{p,q} ~ \alpha_{n_1} \dotsm \alpha_{n_q} = 0.$$
	For instance, the first relations are
	\begin{eqnarray*}
		c_{1,0} + c_{0,1} \alpha_1 & = & 0, \\
		c_{2,0} + c_{1,1} \alpha_1 + c_{0,2} \alpha_1^2 + c_{0,1} \alpha_2 & = & 0.
	\end{eqnarray*}
	Since $P(t,X)$ converges on $D(0,r)$, then we must have $|c_{p,q}| <r<1$ for any $p,q \geq 0$.
	Hence, we obtain $|\alpha_n|< |c_{0,1}|^{-n}$ for any $n \geq 0$, so that the power series $\alpha(t)$ converges in the disk $D(0,\frac{1}{|c_{0,1}|})$.
\end{proof}

We denote by $U := D(0,r) - \mathcal{S}$ the complement of the finite set from Lemma \ref{eigen}.
For any $\zeta \in U$, we                                                          define the function
$\epsilon_\zeta \colon \Sp(A(\zeta)) \to \NN^{p+2}$ as
$$\epsilon_\zeta \colon \alpha \mapsto \(\(\dim_{F_\CC} \(\(E(\zeta)_\alpha \otimes_{\overline{\QQ}} \CC\) \cap \(H_i \otimes_\CC F_\CC\)\)\)_{0 \leq i \leq p}, \rk_{F_{\overline{\QQ}}} \(A(\zeta) -\alpha\)_{|E(\zeta)_\alpha}\),$$
where $E(\zeta)_\alpha \subset \mathrm{Mat}_{n,1}(\overline{\QQ})$ denotes the generalized eigenspace
$$E(\zeta)_\alpha := \ker\(\(A(\zeta)-\alpha\)^m\) ~,~~m>>1.$$
For each $\zeta \in U$, we define the image set
$$\mathcal{A}(\zeta) := \epsilon_\zeta\(\Sp(A(\zeta))\) \subset \NN^{p+2}.$$
By Lemma \ref{rk}, there exists a finite subset $\mathcal{S}' \subset U$ such that for all $\zeta_1, \zeta_2 \in U-\mathcal{S}'$, we have
$$\mathcal{A}(\zeta_1) = \mathcal{A}(\zeta_2) \subset \NN^{p+2}.$$

\subsection{Invariant properties}
We present here the core Properties $\clubsuit$ and $\coeur$ that are used to prove irrationality statements in the last two sections of this paper.
\begin{dfn}[Spectrum]\label{Spectrum}
For any $K$-evaluation map $\ev \colon R^*_{f,K} \to S^*_K$, we define the \textit{spectrum} as
$$\Sp^X_{\ev} := \left\lbrace \alpha \in F_{\overline{\QQ}} ~|~~ \det \(\ev\(\kappa_\tau\) -\alpha b^2\)=0 \right\rbrace \subset F_{\overline{\QQ}},$$
where $\tau := T_0 \alpha_0 + \dotsb + T_h \alpha_h$.
For any $\alpha \in F_{\overline{\QQ}}$, we define
$$E^X_{\ev,\alpha} := \ker \(\ev\(\kappa_\tau\) -\alpha b^2\)^m \subset H^*(X,\overline{\QQ})_{S_{\overline{\QQ}}} ~,~~m>>1.$$
This is the generalized eigenspace when $\alpha \in \Sp_\ev^X$ and it vanishes otherwise.
\end{dfn}

\begin{rem*}
We explain here why we introduce the term $b^2$ in the definition of the spectrum.
If we write the matrix of $\ev(\kappa_\tau)$ in the homogeneous basis $\(\phi_kb^{-\deg(\phi_k)}\)$ of $H^*(X,K)_{S_K}$, then by homogeneity, it is of the form $M \cdot b^2$ where $M$ is a matrix with coefficients in the field $F_K$.
Thus we observe that the eigenvalues of $\ev(\kappa_\tau)$ are of the form $\alpha b^2$ with $\alpha \in F_{\overline{\QQ}}$ being an eigenvalue of $M$.
Furthermore, if $(\gamma_1, \dotsc, \gamma_r)$ is a $F_{\overline{\QQ}}$-basis of a generalized eigenspace of $M$, then we have
$$E^X_{\ev,\alpha} = S_{\overline{\QQ}} \langle \gamma_1, \dotsc, \gamma_r \rangle.$$
\end{rem*}

\begin{rem*}
We use the field $\overline{\QQ}$ of algebraic numbers for convenience of notations here, but it makes no difference to use instead the number field $K'$ that is the splitting field of the polynomial $\det \(\ev\(\kappa_\tau\)b^{-2} - X\) \in F_K[X]$.
\end{rem*}

\begin{rem}[String equation]\label{string}
For any $K$-evaluation function $\ev$ and any element $\lambda \in F_K$, we define the $K$-evaluation function $\ev+\lambda \un_X$ by
$$x \mapsto \left\lbrace \begin{array}{ll}
\ev(x) & \mathrm{if} ~ x \in \left\lbrace \mathfrak{q}', {Q'}^{\beta'}, T_1, \dotsc, T_h \right\rbrace , \\
\ev(T_0)+\lambda b^2 & \mathrm{if}~x=T_0.
\end{array}\right.$$
From this, we obtain
\begin{eqnarray*}
\ev\(\Eu_\tau\) & = & (\ev+\lambda \un_X)\(\Eu_{\tau}\) - \lambda b^2 ~ \un_X, \\
\ev\(\kappa_\tau\) & = & (\ev+\lambda \un_X)\(\kappa_\tau\) - \lambda b^2 ~, \\
\Sp^X_\ev & = & \Sp^X_{\ev+\lambda \un_X}-\lambda ~ \subset F_{\overline{\QQ}},
\end{eqnarray*}
and for any $\alpha \in F_{\overline{\QQ}}$, we obtain
$$E^X_{\ev,\alpha} = E^X_{\ev+\lambda \un_X,\alpha+\lambda} \subset H^*(X,\overline{\QQ})_{S_{\overline{\QQ}}}.$$
\end{rem}

\begin{dfn}\label{int}
For any $K$-evaluation map $\ev\colon R^*_{f,K} \to S^*_K$ and any $\alpha \in F_{\overline{\QQ}}$, we define the following \textit{integers}:
\begin{eqnarray*}
\rho^X_{\ev,\alpha} & = & \rk_{S^*_{\overline{\QQ}}} \(E^X_{\ev,\alpha} \cap \(H^*(X)^\Hdg \otimes_\QQ S^*_{\overline{\QQ}}\)\), \\
\nu^X_{\ev,\alpha} & = & \rk_{S^*_\CC} \(\(E^X_{\ev,\alpha} \otimes_{\overline{\QQ}}\CC\) \cap \(H^{(2)}(X) \otimes_\CC S^*_\CC\)\), \\
{\nu'}^X_{\ev,\alpha} & = & \rk_{S^*_\CC} \(\(E^X_{\ev,\alpha} \otimes_{\overline{\QQ}} \CC\) \cap \(H^{(1)}(X) \otimes_\CC S^*_\CC\)\), \\
\gamma^X_{\ev,\alpha} & = & \rk_{S^*_{\overline{\QQ}}} \(\ev\(\kappa_{\tau}\) -\alpha\)_{|E^X_{\ev,\alpha}}, \\
\end{eqnarray*}
where $\tau := T_0\alpha_0+\dotsb+T_h\alpha_h$.
These integers are zero whenever $\alpha \notin \Sp^X_\ev$.

Furthermore, we say that two $K$-evaluation functions $\ev$ and $\ev'$ are \textit{equivalent}, and we write $\ev \sim \ev'$, when we have a bijection $\Phi \colon \Sp^X_\ev \to \Sp^X_{\ev'}$ and an isomorphism
$$E^X_{\ev,\alpha} \simeq E^X_{\ev',\Phi(\alpha)}$$
that is in the $K$-linear span of morphisms of Hodge structures.
\end{dfn}

\begin{exa}\label{equiv}
By Remark \ref{string}, we have
$$\ev \sim \ev+\lambda\un_X,$$
for any $K$-evaluation function $\ev$ and any $\lambda \in F_K$.

For any $K$-evaluation function $\ev$ and any non-zero element $\lambda \in F_K$, we have
$$\ev \sim \lambda \ev.$$
Indeed, if we write these endomorphisms in the bases $\mathcal{B} = \(\phi_0, \dotsc, \phi_m\)$ and
$$\mathcal{B}' = \(\lambda^{-\deg(\phi_0)}\phi_0, \dotsc, \lambda^{-\deg(\phi_m)}\phi_m\),$$
then we obtain $[\ev]^\mathcal{B}_\mathcal{B}=[\lambda\ev]^{\mathcal{B}'}_{\mathcal{B}'}$.
\end{exa}

We use the numerical data from Definition \ref{int} to produce Properties $\clubsuit$ and $\coeur$.

\begin{dfn}\label{keys2}
Let $f \colon X \to Z$ be a morphism satisfying Assumption \ref{assump}.

\noindent
We say that a $K$-evaluation map $\ev \colon R^*_{f,K} \to S^*_K$ satisfies:
\begin{itemize}
\item \textit{Property $\clubsuit_{f,K}$} if we have
$$\nu^X_{\ev,\alpha} = 0 \quad \textrm{or} \quad \rho^X_{\ev,\alpha} \geq 3$$
for all $\alpha \in \Sp_\ev^X$,
\item \textit{Property $\coeur_{f,K}$} if we have
$$\nu^X_{\ev,\alpha} = 0 \quad \textrm{or} \quad  {\nu'}^X_{\ev,\alpha} \neq 0 \quad \textrm{or} \quad \gamma^X_{\ev,\alpha} \geq 2$$
for all $\alpha \in \Sp_\ev^X$.
\end{itemize}
\end{dfn}

\begin{nt}
Let $\Omega_{f,K}$ be the set of all $K$-evaluation maps $\ev \colon R^*_{f,K} \to S^*_K$.
It can be identified via the image of generators of $R^*_{f,K}$ with a closed analytic subspace of an affine space over the non-archimedean field $F_K$ (note that the power of $b$ can be recovered uniquely from the degree of the given generator, so that it gives no extra information).
In particular, it inherits a topology.
We denote by $U_{f,K} \subset \Omega_{f,K}$ the dense open subset of all $K$-evaluation maps for which the cardinality of $\mathrm{Sp}_\ev^X$ is maximal.
\end{nt}

\begin{rem}
When all the assumptions of Propositions \ref{incr2} and \ref{incr3} are satisfied (e.g.~for the blow-up morphism), then these Propositions preserve the property of being in $U_{\bullet,K}$, i.e.~an evaluation map $\ev \in U_{f,K}$ induces an evaluation map $\ev' \in U_{\id_X,K}$ and conversely.
\end{rem}

\begin{dfn}\label{key}
Let $f \colon X \to Z$ be a morphism satisfying Assumption \ref{assump} and $\diamondsuit \in \left\lbrace \clubsuit, \coeur \right\rbrace$.

\noindent
We say that the variety $X$ satisfies \textit{Property $\diamondsuit_{f,K}$} if
\begin{itemize}
\item $\left\lbrace \ev \in U_{f,K} ~|~~\ev \textrm{ satisfies Property $\diamondsuit_{f,K}$} \right\rbrace$
is dense in $U_{f,K}$.
\end{itemize}

\noindent
We say that the variety $X$ satisfies \textit{Property $\diamondsuit^\mathrm{strong}_{f,K}$} if
\begin{itemize}
\item $\left\lbrace \ev \in \Omega_{f,K} ~|~~\ev \textrm{ satisfies Property $\diamondsuit_{f,K}$} \right\rbrace=\Omega_{f,K}$.
\end{itemize}

\noindent
At last, we simply say that $X$ satisfies Property $\diamondsuit^{(\mathrm{strong})}$ if there exists a number field $K$ such that $X$ satisfies Property $\diamondsuit^{(\mathrm{strong})}_{\id_X,K}$ for the identity map $\id_X \colon X \to X$.
\end{dfn}

\begin{pro}\label{redef}\label{strong}
Property $\clubsuit_{f,K}$ is equivalent to each one of the following equivalent conditions
\begin{itemize}
	\item all $\ev \in U_{f,K}$ satisfy Property $\clubsuit_{f,K}$,
	\item there exists $\ev \in U_{f,K}$ satisfying Property $\clubsuit_{f,K}$,
\end{itemize}
and Property $\coeur_{f,K}$ is equivalent to the condition
\begin{itemize}
	\item there exists $\ev \in U_{f,K}$ satisfying Property $\coeur_{f,K}$.
\end{itemize}
These are the conditions that we will use from now on.
Furthermore, we have
$$\mathrm{Property}~\clubsuit_{f,K}~\iff~\mathrm{Property}~\clubsuit^{\mathrm{strong}}_{f,K}.$$
\end{pro}

\begin{proof}
Let $(\ev_{\zeta})_{\zeta \in D(0,1)}$ be a one-parameter family of evaluation maps such that the corresponding family of matrices $\([\ev_{\zeta}(\kappa_{\tau})]\)_\zeta$ is as in Section \ref{nonarchi}.
Note that, by Lemma \ref{eigen}, either none of these evaluation maps are in $U_{f,K}$ or there is a finite set $\mathcal{S} \subset D(0,1)$ such that $\ev_{\zeta} \in U_{f,K}$ for all $\zeta \in D(0,1)-\mathcal{S}$.
We assume here that we are in the later case, and for any $\zeta \in D(0,1)-\mathcal{S}$, we write by $\alpha(\zeta) \in \mathrm{Sp}^X_{\ev_{\zeta}}$ the one-parameter deformation of an eigenvalue as in Lemma \ref{eigen}.
Then we observe that the function $\zeta \mapsto (\rho^X_{\ev_\zeta,\alpha(\zeta)},\nu^X_{\ev_\zeta,\alpha(\zeta)},{\nu'}^X_{\ev_\zeta,\alpha(\zeta)})$ is constant, whereas the function $\zeta \mapsto \gamma^X_{\ev_\zeta,\alpha(\zeta)}$ is lower semi-continuous, which concludes the proof of the first statements.

For the last statement, we observe that at a value $\zeta_0 \in \mathcal{S}$ some eigenvalues coincide and then the corresponding generalized eigenspaces merge to a sum. Since the numerical conditions defining Property $\clubsuit_{f,K}$ are preserved by taking a sum of vector spaces, the function $\zeta \mapsto (\rho^X_{\ev_\zeta,\alpha(\zeta)},\nu^X_{\ev_\zeta,\alpha(\zeta)},{\nu'}^X_{\ev_\zeta,\alpha(\zeta)})$ is upper semi-continuous in $D(0,1)$, which concludes the proof.
\end{proof}

\begin{rem}\label{incrr}
Above properties do not really depend on the number field.
Specifically, for a number field extension $K_1 \subset K_2$, we have
$$\textrm{Property }\diamondsuit^{(\mathrm{strong})}_{f,K_2} \implies \textrm{Property }\diamondsuit^{(\mathrm{strong})}_{f,K_1}$$
because every $K_1$-evaluation function $\ev_1$ induces a $K_2$-evaluation function $\ev_2$ by extension of scalars and the numerical invariants stay the same.
\end{rem}

In the next Proposition, we discuss the dependance of our Properties on the morphism, see Lemma \ref{ext} for its application to the blow-up map and to the inclusion of a smooth center.

\begin{pro}\label{incr}
Let $\diamondsuit \in \left\lbrace \clubsuit, \coeur \right\rbrace$ and let $f \colon X \to Z$ be a morphism satisfying Assumption \ref{assump} and the condition of Proposition \ref{incr2}.
Then we have
$$\textrm{Property }\diamondsuit^{\mathrm{strong}}_{\id_X,K} \implies \textrm{Property }\diamondsuit^{\mathrm{strong}}_{f,K}$$
and
$$\textrm{Property }\diamondsuit_{f,K} \implies \textrm{Property }\diamondsuit_{\id_X,K}.$$
If furthermore the assumptions of Proposition \ref{incr3} are satisfied as well, together with the following monoid assumption
\begin{itemize}
\item we have a monoid isomorphism $\mathrm{NE}_\NN(X) \simeq \NN \oplus P$, for some monoid $P$, such that the curve class $\beta_0$ from Proposition \ref{incr3} maps to $(1,0)$,
\end{itemize}
then we have
$$\textrm{Property }\diamondsuit_{\id_X,K} \implies \textrm{Property }\diamondsuit_{f,K}.$$
\end{pro}

\begin{proof}
The first two statements are clear because every $K$-evaluation map in $\Omega_{f,K}$ (resp.~in $U_{f,K}$) induces a $K$-evaluation map in $\Omega_{\id_X,K}$ (resp.~in $U_{f,K}$) and they have exactly the same invariants since the endomorphism $\kappa_\tau$ is defined using $R^*(X,K)$.

The third statement is more subtle and requires a deformation argument as in Section \ref{nonarchi}.
By Property $\diamondsuit_{\id_X,K}$, there exists a $K$-evaluation map $\ev' \in U_{\id_X,K}$ satisfying the desired property.
Without any loss of generality, we assume it is normalized.
By the monoid assumption, it can be factored as follows.

Let $t$ be a formal variable and $\ev'_t \colon R^*_{\id_X,K} \to S^*_K[[t]]$
be the morphism given by
$$\ev'_t(Q_X^\beta) = \ev'(Q_X^{\beta'}) t^nb^{\deg(Q_X^{\beta_0})},$$
for any effective curve class $\beta \in \mathrm{NE}_\NN(X)$ where we identify $\beta = (n,p)$ and $\beta'=(n,0)$ by the monoid isomorphism, and equal to $\ev'$ on the remaining generators.
Therefore, we have
$$\ev' = \(\ev'_t\)_{t=\ev'(Q_X^{\beta_0})b^{-\deg(Q_X^{\beta_0})}}.$$
Furthermore, for any $\zeta \in D(0,1)$ (the unit disk in $F_K$), we define a family of evaluation maps $\ev'_\zeta \colon R^*_{\id_X,K} \to S^*_K$ by setting $$\ev'_\zeta:= \(\ev'_t\)_{t=\zeta}.$$

In the homogeneous basis $\mathcal{B}_b:=\(\phi_0 b^{-\deg(\phi_0)}, \dotsc, \phi_m b^{-\deg(\phi_m)}\)$ of $H^*(X,K)_{S_K}$, the endomorphism $\ev'_{1,t}(\kappa_\tau)$ yields a matrix
$$A := b^{-2} \cdot [\ev'_t(\kappa_\tau)]^{\mathcal{B}_b}_{\mathcal{B}_b} \in \mathrm{Mat}(F_K[[t]]).$$
Note that by homogeneity and our choice of the basis, the coefficients of the matrix $A$ do not involve the variable $b$ anymore.
For any $\zeta \in D(0,1)$, the matrix
$$A(\zeta) := (A)_{t=\zeta} \in \mathrm{Mat}(F_K)$$
is well-defined.
By Proposition \ref{span HS}, the endomorphism $\ev'_t(\kappa_\tau)$ is in the $K$-linear span of morphisms of Hodge structures and therefore it stabilizes the subspaces $H^*(X)^\Hdg_S$ and $H^{(k)}(X)_{S_\CC}$.
Thus the matrix $A$ stabilizes the subspaces $H_0:=[H^*(X)^\Hdg_S]_{\mathcal{B}_b}$ and $H_k:=[H^{(k)}(X)_{S_\CC}]_{\mathcal{B}_b}$ of the space of column matrices.

Applying the results of Section \ref{nonarchi}, for any $0 < r < 1$, we find a finite subset $\mathcal{S} \subset D(0,r)$ such that, for any $\zeta \in D(0,r)-\mathcal{S}$, we have $\ev'_\zeta \in U_{\id_X,K}$, and for any $\zeta_1, \zeta_2 \in D(0,r)-\mathcal{S}$, the numerical condition in Property $\diamondsuit_{\id_X,K}$ is satisfied by $\ev'_{\zeta_1}$ if and only if it is satisfied by $\ev'_{\zeta_2}$.
We apply this to $\zeta_1= \ev'(Q_X^{\beta_0})b^{-\deg(Q_X^{\beta_0})}$, for which we have $\ev'_{\zeta_1}=\ev'$ satisying Property $\diamondsuit_{\id_X,K}$, and to $\zeta_2 \in D(0,r)-\mathcal{S}$ such that
the $K$-evaluation map $\ev'_2 := \ev'_{\zeta_2}$ is in $U_{\id_X,K}$ and satisfies the inequality of Proposition \ref{incr3}.
Note that there are infinitely many such values $\zeta_2$.
Consequently, the $K$-evaluation map $\ev'_2$ can be promoted to a $K$-evaluation map $\ev_2 \colon R^*_{f,K} \to S^*_K$ in $U_{f,K}$ by Proposition \ref{incr3} (for some arbitrary choice of the number $0 < \epsilon < 1$).
Furthermore, it satisfies Property $\diamondsuit_{f,K}$ (as $\ev'_2$ does because $\ev'_1$ does), which concludes the proof.
\end{proof}

\subsection{Blow-up formula}\label{sect Bl}
In this section, we study the behaviour of our Properties $\clubsuit$ and $\coeur$ under a blow-up map. The technical core of this section is Theorem \ref{Bl} proven by Iritani in \cite{Iritani}.

Let $X$ be a smooth projective complex variety and $X' \subset X$ be a smooth subvariety of codimension $r \geq 2$.
We denote by $\widetilde{X} = \mathrm{Bl}_{X'} X$ the blow-up of $X$ at the smooth center $X'$ and by $\mathrm{pr} \colon \widetilde{X} \to X$ the projection map.
Recall that we have an isomorphism of Hodge structures
$$H^*(\widetilde{X},\QQ) \simeq H^*(X,\QQ) \oplus \bigoplus_{k=1}^{r-1} H^{*-2k}(X',\QQ)(-k),$$
given by $\mathrm{pr}^*$ on $X$ and by $i_*\circ \(H_E^{k-1} \cup \) \circ \mathrm{pr}_{|E}^*$ on the summands, where $i \colon E \hookrightarrow \widetilde{X}$ is the embedding of the exceptional divisor and $H_E$ is its hyperplane class.
Hence, we can fix graded bases $\mathcal{B}$, $\mathcal{B}'$, and $\widetilde{\mathcal{B}}$ of the cohomologies of $X$, $X'$, and $\widetilde{X}$ such that
$$\widetilde{\mathcal{B}} \simeq \mathcal{B} \sqcup \bigsqcup_{k=1}^{r-1} \mathcal{B}'$$
under the isomorphism. Similarly, we fix graded bases $\mathcal{H}$, $\mathcal{H}'$, and $\widetilde{\mathcal{H}}$ of the Hodge subspaces such that
$$\widetilde{\mathcal{H}} \simeq \mathcal{H} \sqcup \bigsqcup_{k=1}^{r-1} \mathcal{H}'.$$
Let $h$ be the dimension of $H^*(\widetilde{X})^\Hdg$ and introduce formal variables $T_0,\dotsc, T_h$ associated to the elements $\alpha_0,\dotsc,\alpha_h$ of the equivalent bases $\widetilde{\mathcal{H}} \simeq \mathcal{H} \sqcup \bigsqcup_{k=1}^{r-1} \mathcal{H}'$.

Following Iritani \cite{Iritani}, we define the integer
$$s := \left\lbrace \begin{array}{ll}
2(r-1) & \textrm{if $r$ is odd,} \\
r-1 & \textrm{if $r$ is even.}
\end{array}\right.$$
Iritani introduces a ring
$$\widehat{R} := \CC[\mathfrak{q}^{\pm1/s}][[Q]],$$
where $Q$ is the Novikov variable of $X$, equipped with ring maps from the Novikov rings of $X$, $X'$, and $\widetilde{X}$, see \cite[Remark 1.3]{Iritani}.
Furthermore, as in \cite[line above Lemma 3.1]{Iritani}, we assign the degree of $\mathfrak{q}$ to be
$$\deg\(\mathfrak{q}\) = 2\(r-1\).$$
We write $\mathfrak{q}':=\mathfrak{q}^{1/s}$, whose degree is in $\left\lbrace 1,2\right\rbrace$.

\begin{nt}
We introduce the graded $\QQ$-algebra
$$\QQ[\mathfrak{q}^{\pm 1/s}][[Q,T_0,\dotsc,T_h]] =:\widehat{R}^*(X,\QQ) = \bigoplus_{n \in \ZZ} \widehat{R}^n(X,\QQ).$$
According to our previous notation, we have
$$\widehat{R}^*(X,K) := R^*_{\mathrm{pr},K}.$$
Notice that by definition the graded $K$-algebras $R^*_{\id_X,K}$ and $R^*_{X' \subset X,K}$ are almost equal to $\widehat{R}^*(X,K)$ as well; the only difference being that they are obtained by taking a subset of the formal variables $(T_0, \dotsc, T_h)$ corresponding to Hodge classes in $X$ and in $X'$.
Furthermore, by \cite[Equation (1.1)]{Iritani}, we have embeddings of graded $K$-algebras
\begin{eqnarray*}
R^*(X,K) \hookrightarrow \widehat{R}^*(X,K) & & Q \mapsto Q, \\
R^*(\widetilde{X},K) \hookrightarrow \widehat{R}^*(X,K) & & \widetilde{Q}^{\widetilde{\beta}} \mapsto Q^{\mathrm{pr}_*(\widetilde{\beta})} \mathfrak{q}^{-[E] \cdot \widetilde{\beta}},
\end{eqnarray*}
where $E$ is the exceptional divisor in $\widetilde{X}$ and $\mathrm{pr} \colon \widetilde{X} \to X$ is the projection.
In particular, we have
$$\widetilde{Q}^{l_E} \mapsto \mathfrak{q},$$
where $l_E$ is the `vertical' line in the exceptional divisor $E$ defined by $l_E = H_E^{r-2}$.
\end{nt}

\begin{lem}\label{ext}
For each $\diamondsuit \in \left\lbrace \clubsuit, \coeur\right\rbrace$, we have
\begin{eqnarray*}
X' ~ \mathrm{ satisfies } ~ \diamondsuit^{\mathrm{strong}}_{\id_{X'},K} & \implies & X' ~ \mathrm{ satisfies } ~ \diamondsuit^{\mathrm{strong}}_{X' \subset X,K}, \\
\widetilde{X} ~ \mathrm{ satisfies } ~ \diamondsuit_{\id_{\widetilde{X}},K} & \iff & \widetilde{X} ~ \mathrm{ satisfies } ~ \diamondsuit_{\mathrm{pr},K}.
\end{eqnarray*}
\end{lem}

\begin{proof}
We just have to show that all the assumptions of Proposition \ref{incr} are satisfied.

For the variety $X'$, we have the map
$${Q'}^{\beta'} \mapsto Q^{i_*(\beta')} \mathfrak{q}^{-\frac{\int_{\beta'} c_1(\cN_i)}{r-1}}$$
where $\beta' \in \mathrm{NE}_\NN(X')$, $i \colon X' \hookrightarrow X$, and $\cN_i$ its normal bundle, see \cite[Equation 1.1]{Iritani}.
Since $i_*(\beta') \neq 0$ for any non-zero curve class $\beta' \in \mathrm{NE}_\NN(X')$, the condition of Proposition \ref{incr2} is satisfied.
Similarly, for the variety $\widetilde{X}$, the map is
$$\widetilde{Q}^{\widetilde{\beta}} \mapsto Q^{\mathrm{pr}_*(\widetilde{\beta})} \mathfrak{q}^{-[E] \cdot \widetilde{\beta}}.$$
Since $\mathrm{pr}(\widetilde{\beta}) = 0$ implies that $\widetilde{\beta}$ is a positive multiple of $l_E$, the condition of Proposition \ref{incr2} is again satisfied.
Furthermore, the effective curve class $\beta_0:=l_E$ is the vertical line in the exceptional divisor.
We have $j_{\beta_0}=s$ and for any effective curve class $\widetilde{\beta} \in \mathrm{NE}_\NN(\widetilde{X})$ we have $j_{\widetilde{\beta}}=-[E] \cdot \widetilde{\beta}s$, so that the divisibility condition is satisfied as well.
Furthermore, for any curve class $\widetilde{\beta} \in \mathrm{NE}_\NN(\widetilde{X})$, we can write uniquely
$$\widetilde{\beta} = \widetilde{\beta}' + n_{\widetilde{\beta}} l_E,$$
where $\widetilde{\beta}'$ is the strict transform of the effective curve class $\mathrm{pr}_*(\widetilde{\beta}) \in \mathrm{NE}_\NN(X)$ and $n_{\widetilde{\beta}} \in \NN$, giving the monoid isomorphism
$$\mathrm{NE}_\NN(\widetilde{X}) \simeq \NN \oplus \mathrm{NE}_\NN(X),$$
the surjectivity of the map $j \colon \mathrm{NE}_\NN(\widetilde{X}) \to \mathrm{NE}_\NN(X)$, and the injectivity of the morphism $J \colon \QQ[[\widetilde{Q}]] \to \QQ[{\mathfrak{q}'}^{\pm 1}][[Q]]$.
Hence, all the assumptions of Proposition \ref{incr} are satisfied here.
\end{proof}

For any element
$$\tau \in H^2(\widetilde{X})^\Hdg_{R(X,\QQ)} \simeq \(H^*(X)^\Hdg \oplus \bigoplus_{k=1}^{r-1} H^{*-2k}(X')^\Hdg(-k)\)^{*=2}_{R(X,\QQ)},$$
we denote by $\widetilde{\kappa}_\tau$ (resp.~$\kappa_\tau$, $\kappa'_\tau$) the endomorphism associated to $\widetilde{X}$ (resp.~$X$, $X'$) by Definition \ref{matrix}.
Precisely, we have
$$\widetilde{\kappa}_\tau \in \mathrm{End}_{R(\widetilde{X},\QQ)}\(H^*(\widetilde{X},\QQ)_{R(\widetilde{X},\QQ)}\) \subset \mathrm{End}_{\widehat{R}(X,\QQ)^{\mathrm{even}}}\(H^*(\widetilde{X},\QQ)_{\widehat{R}(X,\QQ)}\)$$
and the endomorphism given by the direct sum
$$\(\kappa \oplus \bigoplus_{k=1}^{r-1} \kappa' \)_\tau \in \mathrm{End}_{\widehat{R}(X,\QQ)^{\mathrm{even}}}\(H^*(X,\QQ) \oplus \bigoplus_{k=1}^{r-1} H^{*-2k}(X',\QQ)(-k)\)_{\widehat{R}(X,\QQ)}.$$

\begin{rem}\label{conv ev}
There is a slight abuse of notation here.
If we write the element $\tau = f_0 \alpha_0 + \dotsc + f_h \alpha_h$, then there are subsets
$I_k \subset \left\lbrace 0,\dotsc,h\right\rbrace$
for $k=0, \dotsc, r-1$ such that $\(\alpha_j\)_{j \in I_k}$ corresponds to the basis of the $k$-th copy of $X'$ (or to $X$ for $k=0$).
Setting $\tau_k := \sum_{j \in I_k} f_j \alpha_j$, we use the notation
$$\(\kappa \oplus \bigoplus_{k=1}^{r-1} \kappa'\)_\tau := \kappa_{\tau_0} \oplus \bigoplus_{k=1}^{r-1} \kappa'_{\tau_k}.$$
\end{rem}

\begin{thm}[Iritani, \cite{Iritani}]\label{Bl}
Let $K$ be the number field obtained by adding the $s$-th roots of unity, the imaginary unit $\ci$, and the square root $\sqrt{r-1}$.
Furthermore, let
$$K' := \left\lbrace \begin{array}{ll}
K(\pi) & \textrm{if $\dim_\CC(X') \geq 3$,} \\
K & \textrm{if $\dim_\CC(X') \leq 2$.}
\end{array}\right. $$
Then there exist formal power series
$$f(T_0), \dotsc, f(T_h) \quad \textrm{and} \quad \widetilde{f}(T_0), \dotsc, \widetilde{f}(T_h)$$
in $\widehat{R}^*(X,K)$, such that the functions $f$ and $\widetilde{f}$ are homogeneous of degree zero and satisfy the following.
Defining $f$ and $\widetilde{f}$ as the identity on Novikov variables and using Remark \ref{change of variables}, we produce two morphisms of graded $K$-algebras of degree zero:
\begin{itemize}
\item $g \colon \widehat{R}^*(X,K) \to \widehat{R}^*(X,K)$, $T_k \mapsto g_k$,
\item $\widetilde{g} \colon \widehat{R}^*(X,K) \to \widehat{R}^*(X,K)$, $T_k \mapsto \widetilde{g}_k$.
\end{itemize}
Then there exist two isomorphisms $\Psi$ and $\widetilde{\Psi}$
$$\Psi, \widetilde{\Psi} ~~\colon \(H^*(X,K') \oplus \bigoplus_{k=1}^{r-1} H^{*-2k}(X',K')\)_{\widehat{R}(X,K')} \to H^*\(\widetilde{X},K'\)_{\widehat{R}(X,K')}$$
such that
\begin{eqnarray*}
\(\kappa \oplus \bigoplus_{k=1}^{r-1} \kappa' \)_\tau & = & \Psi^{-1} \circ \widetilde{\kappa}_{g(\tau)} \circ \Psi, \\
\(\kappa \oplus \bigoplus_{k=1}^{r-1} \kappa' \)_{\widetilde{g}(\tau)} & = & \widetilde{\Psi}^{-1} \circ \widetilde{\kappa}_\tau \circ \widetilde{\Psi},
\end{eqnarray*}
where $\tau :=T_0 \alpha_0+\dotsb+T_h \alpha_h$.
Moreover, the isomorphisms $\Psi$ and $\widetilde{\Psi}$ are in the $K'$-span of morphisms of Hodge structures, i.e.~there are finitely many morphisms of free $\widehat{R}(X,\QQ)^{\mathrm{even}}$-modules and of Hodge structures
$$\Psi_1, \dotsc, \Psi_N \colon \(H^*(X,\QQ) \oplus \bigoplus_{k=1}^{r-1} H^{*-2k}(X',\QQ)\(-k\)\)_{\widehat{R}(X,\QQ)} \to H^*\(\widetilde{X},\QQ\)_{\widehat{R}(X,\QQ)}$$
such that
$$\Psi = s_1 \Psi_1 + \dotsb + s_N \Psi_N \textrm{ with $s_1, \dotsc, s_N \in K'$,}$$
and similarly for $\widetilde{\Psi}$.
\end{thm}

\begin{proof}
The proof follows from a careful reading of the paper of Iritani \cite{Iritani}. We provide the necessary details here, following the notation established in that paper.

From Section $2$ through Section $4.1$, all results remain valid when the coefficients $\CC$ are replaced by $\QQ$.
In particular, the discrete Fourier transform is defined over $\QQ$.
The first instance where a larger field is required is in Section $4.2$, specifically in Equation $(4.5)$.
In particular, the continuous Fourier transform is not defined over $\QQ$, and Corollary $4.9$ necessitates an extension of $\QQ$.

In Section $4.3$, the discrete and continuous Fourier transforms are compared.
Corollary $4.11$ holds with $\QQ$ coefficients because it is eventually phrased in terms of the discrete Fourier transform.

In Section $5$, one may replace $\CC$ with $\QQ$ up to Remark $5.6$.
At this stage, it is important to note that results concerning $X$ and $\widetilde{X}$ are defined over $\QQ$, whereas results concerning $Z$ are defined over $\CC$.

Moreover, since the Kirwan map in Section $3.6$ is defined using pull-backs, it is a morphism of Hodge structures.
As a consequence, the maps denoted $\widehat{\mathrm{FT}}_{\widetilde{X}}$ and $\widehat{\mathrm{FT}}_X$ are defined over $\QQ$ and are morphisms of Hodge structures.
The remaining difficulty lies with the map $\widehat{\mathrm{FT}}_{Z,j}$.

The center $Z$ is responsible for both the fractional powers of $\mathfrak{q}$ and the extension of scalars from $\QQ$.
However, the continuous Fourier transform does not require the full field of complex numbers.
From the formulae for $\cF_{F,0}$ near Equation $(4.9)$, for $\Phi(u)$ and for $g(u,\lambda_0)$ in the previous page, and for the quantum Riemann--Roch operator in Equation $(2.24)$, we observe that the operator $\cF_{F,0}$ is defined using only algebraic cycles, $s$-th roots of unity, the imaginary unit $\ci$, and the square root $\sqrt{r-1}$ (note that $c_F=1-r$ for $F=Z$).
One then defines the continuous Fourier transform $\cF_{F,j}$ by
$$\cF_{F,j} = {\cF_{F,0}}_{|\mathfrak{q} \mapsto e^{2 \pi \ci j} \mathfrak{q}}\quad \textrm{(for $F=Z$, we have $S_F=\mathfrak{q}$)}.$$
Note that the notation $e^{2 \pi \ci j}$ is somewhat symbolic here; it should not be replaced by the number $1$ because $\cF_{F,0}$ involves $s$-th roots of $\mathfrak{q}$. It is more precise to write the substitution as $\mathfrak{q}^{1/s} \mapsto e^{2 \pi \ci j/s} \mathfrak{q}^{1/s}$.
Furthermore, $\mathfrak{q}$ appears only in the term $\lambda_0$ defined after Equation $(4.8)$.
Attention must be paid to the term
$$\lambda_0^{\frac{1}{2}-\sum_{\alpha} \(\frac{\rho_\alpha}{z} + \frac{r_\alpha}{2}\)} = \lambda_0^{- \frac{c_1\(\cN_{Z \subset W}\)}{z} - \frac{r}{2}},$$
where $\cN_{Z \subset W}$ is the normal bundle of the blow-up center $Z$ in the space $W:=\mathrm{Bl}_{Z \times 0} \(X \times \PP^1\)$.
When rescaling $\mathfrak{q} \mapsto e^{2 \pi \ci j} \mathfrak{q}$ to define $\cF_{F,j}$, one must introduce
$$e^{\frac{2 \pi \ci j}{r-1} \cdot \(\frac{c_1\(\cN_{Z \subset W}\)}{z} + \frac{r}{2}\)},$$
and specifically the expansion
$$\sum_{n \geq 0} \(\frac{2 \pi \ci j}{r-1} \cdot c_1\(\cN_{Z \subset W}\)\)^n \cfrac{z^{-n}}{n!}.$$
This is an algebraic cycle, but it involves the transcendental number $\pi$.
For $n=1$, we obtain a divisor class which can be transformed into an $s$-th root of unity via Example \ref{divisor rescaling}; thus only terms with $n \geq 2$ create difficulties.

To summarize, the continuous Fourier transform, and thus the main Theorem $5.18$, holds over the field extension $K(\pi)$, where $K$ is the number field containing $s$-th roots of unity, $\ci$, and $\sqrt{r-1}$.
Since the continuous Fourier transform involves only algebraic cycles (via the quantum Riemann--Roch operator), it lies in the span of morphisms of Hodge structures after extending coefficients to $K(\pi)$.
To be precise, the fact that it can be written as a \textit{finite} linear combination of such morphisms follows because $K$ is a finite extension of $\QQ$ and only finitely many powers of $\pi$ are used.
Let us explain this last statement.

First, the number $\pi^n$ always appears with the variable $z^{-n}$ of degree $-2n$. By Example \ref{divisor rescaling}, we may assume $n \geq 2$.
In Section $5.8$, $\Psi$ and $\varsigma$ are reconstructed from their initial data at $Q=\widetilde{\tau}=0$.
By Theorem $5.18$, precisely Equations $(4)$ and $(6)$ there, one observes that $\pi$ appears in $h_{Z,j}$ (see Equation $(5.19)$) and potentially in the $O\(\mathfrak{q}^{-\frac{1}{r-1}}\)$ term.
The former can be handled by Example \ref{divisor rescaling}.
For the latter, a degree argument suffices.
Additionally, regarding $\varsigma^\circ$, one may observe from Equation $(5.45)$ that only the coefficient of $z^{-1}$ appears.
For $\Psi^\circ$, the $Z$-components are homogeneous of degree $-r$, matching the degree of $q_{Z,j}$, so that the term $O\(\mathfrak{q}^{-\frac{1}{r-1}}\)$ has to be homogeneous of degree $0$.
Since the degrees of $\mathfrak{q}^{-\frac{1}{r-1}}$  and $z^{-1}$ are both $-2$, a term involving $\pi^n z^{-n}$ for $n \geq 2$ must carry a cohomology class of degree $2(n+1)$.
Consequently, the power $n$ of $\pi$ is bounded by the complex dimension of the blow-up center $Z$.
In particular, if $\dim_\CC Z \leq 2$, then the cohomology degree is at most $4$, and $\Psi^\circ$ is defined without $\pi$.

Finally, one uses the first equation on page $62$ to reconstruct $\Psi$ and $\widetilde{\tau}$ from the initial data.
This equation involves the matrix $M$ (defined over $\QQ$), $\lambda_j$ (requiring $s$-th roots of unity), and the initial data.
As a consequence, we see that only finitely many powers of $\pi$ are involved in $\Psi$ and $\widetilde{\tau}$, and $\pi$ is not required when $\dim_\CC(Z) \leq 2$.
In particular, in that case at least, one obtains that the main Theorem $5.18$ holds by working over the number field $K$.

In any case, the isomorphisms $\Psi$ and $\widetilde{\Psi}$ are in the $K'$-span of morphisms of Hodge structures, and the variable changes $g$ and $\widetilde{g}$ are defined over $K$, preserving Hodge classes by Proposition \ref{span HS}.
\end{proof}

\begin{rem}
Combining the result of Iritani's recent note \cite{Iritaninote} with the proof above, we see that Theorem \ref{Bl} holds over the number field $K$ obtained by only adjoining $s$-th roots of unity and the imaginary unit $\ci$.
In particular, it is not necessary to add the square root $\sqrt{r-1}$ and it is not necessary to add the transcendantal number $\pi$, i.e.~we have $K'=K$ regardless of the dimension.
\end{rem}

\begin{cor}\label{equality2}
Assume one of the following:
\begin{itemize}
\item let $\ev \colon \widehat{R}^*(X,K) \to S^*_K$ be a $K$-evaluation map in $\Omega_{\mathrm{pr},K}$ and denote by
$$\widetilde{\ev} \colon \widehat{R}^*(X,K) \to S^*_K$$
the composition $\ev \circ g$, which is also a $K$-evaluation map in $\Omega_{\mathrm{pr},K}$,
\item let $\widetilde{\ev} \colon \widehat{R}^*(X,K) \to S^*_K$ be a $K$-evaluation map in $\Omega_{\mathrm{pr},K}$ and denote by
$$\ev \colon \widehat{R}^*(X,K) \to S^*_K$$
the composition $\widetilde{\ev} \circ \widetilde{g}$, which is also a $K$-evaluation map in $\Omega_{\mathrm{pr},K}$,
\end{itemize}
where $K$ is the number field from Theorem \ref{Bl}, and $g$, $\widetilde{g}$ are the changes of variables defined therein.
Then there exist two isomorphisms
$$\textrm{$\Psi$ and $\widetilde{\Psi}$} \quad \colon \(H^*(X,K) \oplus \bigoplus_{k=1}^{r-1} H^{*-2k}(X',K)\)_{S_{K}} \to H^*\(\widetilde{X},K\)_{S_{K}}$$
such that
\begin{eqnarray*}
\ev\(\(\kappa \oplus \bigoplus_{k=1}^{r-1} \kappa' \)_\tau\) & = & \Psi^{-1} \circ \widetilde{\ev}\(\widetilde{\kappa}_\tau\) \circ \Psi, \\
\ev\(\(\kappa \oplus \bigoplus_{k=1}^{r-1} \kappa' \)_\tau\) & = & \widetilde{\Psi}^{-1} \circ \widetilde{\ev}\(\widetilde{\kappa}_\tau\) \circ \widetilde{\Psi},
\end{eqnarray*}
where $\tau :=T_0 \alpha_0+\dotsb+T_h \alpha_h$.
Moreover, the isomorphisms $\Psi$ and $\widetilde{\Psi}$ lie in the $K$-span of morphisms of Hodge structures
$$\(H^*(X,\QQ) \oplus \bigoplus_{k=1}^{r-1} H^{*-2k}(X',\QQ)\(-k\)\)_S \to H^*\(\widetilde{X},\QQ\)_S.$$
In particular, we have the following equality in $F_{\overline{\QQ}}$:
$$\Sp_{\widetilde{\ev}}^{\widetilde{X}} = \Sp_{\ev_0}^X \cup \bigcup_{1 \leq k <r} \Sp_{\ev_k}^{X'},$$
where the $K$-evaluation maps $\ev_0 \in \Omega_{\id_X,K}$ on $X$ and $\ev_k \in \Omega_{X' \subset X,K}$ on the $k$-th copy of $X'$ are induced by $\ev$, see Remark \ref{conv ev}.
Furthermore, for any $\alpha \in F_{\overline{\QQ}}$, we obtain the equalities
$$\epsilon^{\widetilde{X}}_{\widetilde{\ev},\alpha} = \epsilon^{X}_{\ev_0,\alpha} + \sum_{k=1}^{r-1} \epsilon^{X'}_{\ev_k,\alpha}~,~~\textrm{for $\epsilon \in \left\lbrace \rho, \gamma, \nu, \nu' \right\rbrace$.}$$
\end{cor}

\begin{rem*}
In Corollary \ref{equality2}, given $K$-evaluation maps $\ev$ and $\widetilde{\ev}$, the compositions $\ev \circ g$ and $\widetilde{\ev} \circ g$ are degree-zero homogeneous morphisms of graded $K$-algebras, but it is not obvious that they are $K$-evaluation maps.
Indeed, the second statement in Proposition \ref{conv} does not guarantee the existence of a number $0<\epsilon<1$ such that, for any non-zero curve class $\beta \in \mathrm{NE}_\NN(X)$, the inequality
$$0 \leq |\ev(g(Q^\beta))|<\epsilon^{\int_{\beta} \omega} < 1$$
holds (similarly with $\widetilde{\ev}$).
However, in the special situation of Corollary \ref{equality2}, it does hold because the morphism $g$ provided by Theorem \ref{Bl} is just the identity on the Novikov variable.
\end{rem*}

\begin{cor}\label{generic}
We have the equivalence
$$\widetilde{\mathrm{ev}} \in U_{\mathrm{pr},K} \iff
\left\lbrace \begin{array}{l}
\mathrm{ev}_0 \in U_{\id_X,K} \textrm{ and } \mathrm{ev}_k \in U_{X' \subset X,K} \textrm{ for each $k$}, \\
\textrm{the spectra $\left\lbrace \Sp_{\ev_0}^X, \Sp_{\ev_k}^{X'}\right\rbrace_k$ have empty two-by-two intersections.}
\end{array}
\right. $$
In particular, we then have
$$\Sp_{\widetilde{\ev}}^{\widetilde{X}} = \Sp_{\ev_0}^X \sqcup \bigsqcup_{1 \leq k <r} \Sp_{\ev_k}^{X'}$$
and $\epsilon^{\widetilde{X}}_{\widetilde{\ev},\alpha} = \epsilon^{X}_{\ev_0,\alpha}$ or $\epsilon^{X'}_{\ev_k,\alpha}$, for $\epsilon \in \left\lbrace \rho, \gamma, \nu, \nu' \right\rbrace$.
\end{cor}

\begin{proof}
It follows from Corollary \ref{equality2}.
We start with the direction $\implies$.
If $\Sp^X_{\mathrm{ev}_0}$ or one of the $\Sp^X_{\mathrm{ev}_k}$ is not maximal, then we can take another $K$-evaluation map with a bigger spectrum and it will induce another $K$-evaluation map for $\widetilde{X}$ with a bigger spectrum as well.
If some intersection $\Sp_{\ev_{k_1}}^{X'} \cap \Sp_{\ev_{k_2}}^{X'}$ (or similarly $\Sp_{\ev_{k_1}}^{X'} \cap \Sp_{\ev_0}^X$) is non-empty, then we shift the $K$-evaluation map $\ev_{k_1}$ by $\lambda \un_{X'}$ as in Remark \ref{string} by some $\lambda \in F_K$ such that
$$\Sp_{\ev_{k_1}+\lambda \un_{X'}}^{X'} \cap \Sp_{\ev_{k_2}}^{X'} = (\Sp_{\ev_{k_1}}^{X'}+\lambda) \cap \Sp_{\ev_{k_2}}^{X'} = \emptyset.$$
Thus the induced $K$-evaluation map $\widetilde{\mathrm{ev}}_\lambda$ satisfies
$$\mathrm{card} \(\Sp_{\widetilde{\ev}_\lambda}^{\widetilde{X}}\) > \mathrm{card} \(\Sp_{\widetilde{\ev}}^{\widetilde{X}}\).$$
Conversely, regarding the implication $\impliedby$, let us assume that $\Sp_{\widetilde{\ev}}^{\widetilde{X}}$ is not maximal and that the two-by-two intersections of spectra are empty.
Then we can choose another $K$-evaluation map $\widetilde{\ev}'$ with a bigger spectrum and such that the induced $K$-evaluation maps $\ev'_0$ and $\ev'_k$ also have empty two-by-two intersections of spectra (it is possible up to a shift by the unit class $\un_{X'}$).
Then we obtain
$$\mathrm{card} \(\Sp_{\ev'_0}^X\) + \sum_{1 \leq k <r} \mathrm{card} \(\Sp_{\ev'_k}^{X'}\) > \mathrm{card} \(\Sp_{\ev_0}^X\) + \sum_{1 \leq k <r} \mathrm{card} \(\Sp_{\ev_k}^{X'}\)$$
and thus there exists one $0 \leq k <r$ such that $\mathrm{card} \(\Sp_{\ev'_k}^{X^{(')}}\) >\mathrm{card} \(\Sp_{\ev_k}^{X^{(')}}\)$.
\end{proof}

Now, we gather all the ingredients in order to prove the invariance of our Properties under a blow-up map.

\begin{pro}\label{Bl equiv}
Let $\diamondsuit \in \left\lbrace \clubsuit, \coeur\right\rbrace$ and let $K$ be a number field containing all $s$-th roots of unity and the imaginary unit $\ci$.
Then we have the equivalence
\begin{center}
	\begin{tikzpicture}
	\node[
	draw=black,       
	fill=gray!5,              
	thick,                    
	rounded corners=4pt,      
	inner sep=10pt            
	] 
	{
		$\displaystyle 
		\widetilde{X} \text{ satisfies Property } \diamondsuit_{\id_{\widetilde{X}},K} \iff 
		\begin{cases}
		X \text{ satisfies Property } \diamondsuit_{\id_X,K} \text{ and} \\
		X' \text{ satisfies Property } \diamondsuit_{X' \subset X,K}.
		\end{cases}$
	};
	\end{tikzpicture}
\end{center}

\noindent
Note as a reminder that we have already proven:
$$X' ~\textrm{satisfies Property}~ \diamondsuit^{\mathrm{strong}}_{\id_{X'},K} \implies X' ~\textrm{satisfies Property}~ \diamondsuit_{X' \subset X,K}.$$
Furthermore, we also have
$$\widetilde{X} ~ \textrm{satisfies Property} ~ \diamondsuit^{\mathrm{strong}}_{\id_{\widetilde{X}},K}  \implies X ~ \textrm{satisfies Property} ~ \diamondsuit^\mathrm{strong}_{\id_X,K}$$
and
$$\left. \begin{array}{r}
X ~ \textrm{satisfies Property} ~ \diamondsuit^\mathrm{strong}_{\id_X,K} \\
\textrm{and}~X' ~\textrm{satisfies Property}~ \diamondsuit^\mathrm{strong}_{\id_{X'},K}
\end{array}
\right\rbrace
\implies \widetilde{X} ~ \textrm{satisfies Property} ~ \diamondsuit^{\mathrm{strong}}_{\mathrm{pr},K}$$
\end{pro}

\begin{proof}
Consider a $K$-evaluation map $\widetilde{\ev} \in U_{\mathrm{pr},K}$.
By Corollaries \ref{equality2} and \ref{generic}, we obtain $\ev \in U_{\id_X,K}$ and $\ev'_k \in U_{X' \subset X,K}$, for each $1 \leq k < r$, such that
$$\Sp^{\widetilde{X}}_{\widetilde{\ev}} = \Sp^X_{\ev} \sqcup \bigsqcup_{k=1}^{r-1} \Sp^{X'}_{\ev'_k}.$$
From the equalities in Corollary \ref{generic} and by Lemma \ref{ext}, it follows that
$$\textrm{$\widetilde{X}$ sat.~}\diamondsuit_{\id_{\widetilde{X}},K} \implies \textrm{$\widetilde{X}$ sat.~}\diamondsuit_{\mathrm{pr},K}
\implies
\(\textrm{$X'$ sat.~$\diamondsuit_{X' \subset X,K}$ and $X$ sat.~$\diamondsuit_{\id_X,K}$}\).$$

Conversely, let $\ev \in U_{\id_X,K}$ and $\ev' \in U_{X' \subset X,K}$.
By Corollaries \ref{equality2} and \ref{generic}, we obtain $\widetilde{\ev} \in U_{\mathrm{pr},K}$ and from the equalities in Corollary \ref{equality2} and by Lemma \ref{ext}, it follows that
$$\(\textrm{$X'$ sat.~$\diamondsuit_{X' \subset X,K}$ and $X$ sat.~$\diamondsuit_{\id_X,K}$}\) \implies \textrm{$\widetilde{X}$ sat.~}\diamondsuit_{\mathrm{pr},K} \implies \textrm{$\widetilde{X}$ sat.~}\diamondsuit_{\id_{\widetilde{X}},K}.$$

Let us prove now the last statements.
We first assume that $\widetilde{X}$ satisfies Property $\diamondsuit^{\mathrm{strong}}_{\id_{\widetilde{X}},K}$ and thus $\diamondsuit^{\mathrm{strong}}_{\mathrm{pr},K}$ by Proposition \ref{incr}.
Let $\ev \in \Omega_{\id_X,K}$ and $\alpha \in \Sp^X_\ev$.
Choose an arbitrary $K$-evaluation map $\ev' \in \Omega_{X' \subset X,K}$ such that $\alpha \notin \Sp^{X'}_{\ev'}$ (as explained above, it is possible up to a shift by a multiple of the unit).
Then by Corollary \ref{equality2} we obtain $\widetilde{\ev} \in \Omega_{\mathrm{pr},K}$ and
$$\EE^X_{\ev,\alpha} \simeq \EE^{\widetilde{X}}_{\widetilde{\ev},\alpha}$$
with the same numerical invariants.
Since $\widetilde{X}$ satisfies Property $\diamondsuit^{\mathrm{strong}}_{\mathrm{pr},K}$, the $K$-evaluation map $\widetilde{\ev}$ satisfies Property $\diamondsuit_{\mathrm{pr},K}$ as well, and so does $\ev$.

Conversely, let $\widetilde{\ev} \in \Omega_{\mathrm{pr},K}$ and denote by $\ev \in \Omega_{\id_X,K}$ and $\ev'_k \in \Omega_{X' \subset X,K}$ the $K$-evaluation maps induced by Corollary \ref{equality2}, with $1 \leq k < r$.
Since $X'$ satisfies Property $\diamondsuit^{\mathrm{strong}}_{\id_{X'},K}$, then it also satisfies Property $\diamondsuit^{\mathrm{strong}}_{X' \subset X,K}$ by Lemma \ref{ext}.
In particular, the $K$-evaluation maps $\ev \in \Omega_{\id_X,K}$ and $\ev'_k \in \Omega_{X' \subset X,K}$ satisfy the numerical criteria, and so does $\widetilde{\ev} \in \Omega_{\mathrm{pr},K}$ by looking at the equalities in Corollary \ref{equality2}.
\end{proof}

\begin{rem*}
When $\diamondsuit=\coeur$, we highlight the last implication in Proposition \ref{Bl equiv} stops at $\coeur^{\mathrm{strong}}_{\mathrm{pr},K}$, which has the immense drawback of relying on the blow-up map and not on the variety $\widetilde{X}$ alone.
It would be far better to conclude $\coeur^{\mathrm{strong}}_{\id_{\widetilde{X}},K}$ instead and it appears as a limit of this property.
Fortunately, in the case when $X$ is a surface, we will be able to conclude that $\coeur^{\mathrm{strong}}_{\id_{\widetilde{X}},K}$ actually holds, see Proposition \ref{strongsurface}.
\end{rem*}

\begin{dfn}[Weak factorization]\label{weak}
Let $X$ and $Y$ be two birational smooth projective varieties.
A \textit{weak factorization} between $X$ and $Y$ is a sequence of blow-ups and blow-downs at smooth centers
$$X=:X_0 \leftrightarrow X_1 \leftrightarrow X_2 \leftrightarrow X_3 \leftrightarrow \dotsc \leftrightarrow X_q:=Y,$$
where each arrow $X_{i-1} \leftrightarrow X_i$ represents either a blow-up $X_{i-1} \leftarrow X_i$ or a blow-down $X_{i-1} \to X_i$.
Let $Y_i$ denote the smooth center of the birational map $X_{i-1} \leftrightarrow X_i$.
Specifically, if $X_{i-1} \leftarrow X_i$, then $X_i = \mathrm{Bl}_{Y_i}(X_{i-1})$; if $X_{i-1} \to X_i$, then $X_{i-1} = \mathrm{Bl}_{Y_i}(X_i)$.
\end{dfn}

\begin{rem}
If two smooth projective varieties $X$ and $Y$ are birational, then there exists a weak factorization between them, see \cite{paper81}.
\end{rem}

\begin{cor}\label{main}
Let $\diamondsuit \in \left\lbrace \clubsuit, \coeur\right\rbrace$ and let $X$ and $Y$ be two birational smooth projective varieties.
Define the integer
$$s_{\max} := \mathrm{lcd}\(1, \dotsc,\dim(X)-1\),$$
and the number field $\QQ_\ext$ obtained by adjoining all $s_{\max}$-th roots of unity and the imaginary unit $\ci$.
Let $K$ be any number field containing $\QQ_\ext$.

If there exists a weak factorization between $X$ and $Y$ for which all centers $Y_i$ satisfy Property $\diamondsuit^{\mathrm{strong}}_{\id_{Y_i},K}$, then
$$X ~ \textrm{satisfies Property} ~ \diamondsuit_{\id_X,K}  \iff  Y ~ \textrm{satisfies Property} ~ \diamondsuit_{\id_Y,K}.$$

Contrapositively, if $Y$ satisfies Property $\diamondsuit_{\id_Y,K}$ but $X$ fails Property $\diamondsuit_{\id_X,K}$, then for every weak factorization, there is at least one blow-up center $Y_i$ that fails Property $\diamondsuit^{\mathrm{strong}}_{\id_{Y_i},K}$.
\end{cor}

\begin{rem}\label{equality3}
Let
$$X_0 \leftrightarrow X_1 \leftrightarrow X_2 \leftrightarrow X_3 \leftrightarrow \dotsc \leftrightarrow X_q$$
be a weak factorization with blow-up centers $Y_1, \dotsc, Y_q$.
Let $\ev_0$ be an evaluation map for $X_0$ and $\alpha \in \Sp^{X_0}_{\ev_0,\alpha}$ an eigenvalue.
By applying Corollary \ref{equality2} along the weak factorization, we obtain $K$-evaluation maps $\ev_i$ for $X_i$ and $\ev'_{i,k}$ for $(Y_i \to X_j)$, where $1 \leq i \leq q$, $j \in \left\lbrace i-1, i\right\rbrace$ and $1 \leq k < \mathrm{codim}(Y_i)$.
Specifically, for a map in the direction $X_{i-1} \leftarrow X_i$, we have the freedom to choose the evaluation maps for $Y_i$. We choose them such that the two-by-two intersections of spectra $\Sp^{X_{i-1}}_{\ev_{i-1}}$, $\Sp^{Y_i}_{\ev'_{i,k}}$ are all empty.
Consequently, we have
\begin{itemize}
\item $\Sp^{X_{i-1}}_{\ev_{i-1}} \sqcup \bigsqcup_k \Sp^{Y_i}_{\ev'_{i,k}} = \Sp^{X_i}_{\ev_i}$ if $X_{i-1} \leftarrow X_i$,
\item $\Sp^{X_i}_{\ev_i} \cup \bigcup_k \Sp^{Y_i}_{\ev'_{i,k}} = \Sp^{X_{i-1}}_{\ev_{i-1}}$ if $X_{i-1} \to X_i$.
\end{itemize}
From this, we deduce that
$$\epsilon^{X_0}_{\ev_0,\alpha} = \epsilon^{X_q}_{\ev_q,\alpha} + \sum_{(i,k) \in I} \epsilon^{Y_i}_{\ev'_{i,k},\alpha},$$
where $\epsilon \in \left\lbrace \rho, \gamma, \nu, \nu'\right\rbrace$ and $I := \left\lbrace 1 \leq i \leq q, k \in \NN ~|~~X_{i-1} \to X_i \textrm{ and } 1 \leq k \leq \mathrm{codim}_{X_i}Y_i \right\rbrace$, generalizing the final equality of Corollary \ref{equality2}.
Furthermore, if we have an eigenvalue
$$\alpha \in \Sp^{X_0}_{\ev_0} \cap \Sp^Z_{\ev_Z}~,~~ \(Z,\ev_Z\) \in \left\lbrace \(Y_i,\ev'_{i,k}\)_{(i,k) \in I}, \(X_q,\ev_q\) \right\rbrace,$$
there exists an embedding $E^Z_{\ev_Z,\alpha} \hookrightarrow E^{X_0}_{\ev_0,\alpha}$
of $S_{\overline{\QQ}}$-modules.
We also have a $\QQ$-basis $(s_1,\dotsc,s_N)$ of the number field $\QQ_\ext$ and morphisms of $S$-modules and Hodge structures
$$f_1,\dotsc,f_N \colon H^*(Z,\QQ)_S \to H^{*+2k}(X_0,\QQ)_S(k),$$
where $k=0$ if $\ev_Z=\ev_q$ and $k$ is the index in $\ev_Z=\ev_{i,k}$ otherwise, such that the morphism $f_{\QQ_\ext}:=s_1 f_1 + \dotsb + s_N f_N$ is induced from Corollary \ref{equality2} and its extension $f_{\overline{\QQ}}$ restricted to $E^Z_{\ev_Z,\alpha}$ coincides with the embedding $E^Z_{\ev_Z,\alpha} \hookrightarrow E^{X_0}_{\ev_0,\alpha}.$
Finally, the morphism
$$f_{\QQ_\ext} \colon H^*(Z,\QQ_\ext)_{S_{\QQ_\ext}} \to H^{*+2k}(X_0,\QQ_\ext)_{S_{\QQ_\ext}}$$
satisfies the commutation relation
$$\ev_0\(\kappa^{X_0}_\tau\) \circ f = f \circ \ev_Z\(\kappa^Z_\tau\).$$
\end{rem}

\section{Examples regarding Property $\clubsuit$}
\subsection{Basic cases}
\begin{pro}
Let $X$ be a point, a curve, a surface with vanishing Hodge number $h^{2,0}=0$, or a projective space $\PP^n$.
Then, for any number field $K$, the variety $X$ satisfies Property $\clubsuit^{\mathrm{strong}}_{\id_X,K}$.
\end{pro}

\begin{proof}
This follows from the fact that $h^{2,0}=0$ in all these cases; consequently, the integer $\nu^X_{\ev,\alpha}$ is always zero.
\end{proof}

\subsection{Surface with non-negative canonical class}
\begin{pro}\label{nilp}
Let $\Sigma$ be a surface with $K_\Sigma \geq 0$.
Then, for any number field $K$, the variety $\Sigma$ satisfies Property $\clubsuit^{\mathrm{strong}}_{\id_\Sigma,K}$.
\end{pro}

\begin{proof}
Recall from Example \ref{Knef} that the matrix $\left[\kappa_\tau-f_0\right]^{\mathcal{B}}_{\mathcal{B}}$ is nilpotent.
Thus, for any $K$-evaluation map $\ev$, we have
\begin{eqnarray*}
\Sp^\Sigma_{\ev-\ev(T_0)b^{-2}\un_\Sigma} & = & \left\lbrace 0 \right\rbrace, \\
E^\Sigma_{{\ev-\ev(T_0)b^{-2}\un_\Sigma}, 0} & = & H^*(\Sigma,K)_{S_K}.
\end{eqnarray*}
As a consequence, we obtain
$$\rho^\Sigma_{\ev,\ev(T_0)b^{-2}} = \rho^\Sigma_{\ev-\ev(T_0)b^{-2}\un_\Sigma,0} \geq 3,$$
since $H^*(\Sigma,\CC)$ contains at least the Hodge classes $\un_\Sigma, H$, and $H^2$.
\end{proof}

\subsection{Low dimension}
\begin{pro}\label{two}
Every smooth projective variety $X$ of dimension at most $2$ satisfies Property $\clubsuit^{\mathrm{strong}}_{\id_X,K}$ for any number field $K$ containing the imaginary unit $\ci$.
\end{pro}

\begin{proof}
It remains to verify the case of a surface $X$ with $h^{2,0} \neq 0$.
Such a surface is birational to a surface $\Sigma$ with $K_\Sigma \geq 0$, which satisfies Property $\clubsuit^{\mathrm{strong}}_{\id_\Sigma,K}$ by Proposition \ref{nilp}.
Since all blow-up centers in a weak factorization between $X$ and $S$ are points, they satisfy Property $\clubsuit^{\mathrm{strong}}_{\id_\pt,K}$.
By Corollary \ref{main}, with $\QQ_\ext=\QQ(\ci)$, we deduce that $X$ satisfies Property $\clubsuit_{\id_X,K}$, which is equivalent to Property $\clubsuit^{\mathrm{strong}}_{\id_X,K}$ by Proposition \ref{strong}.
\end{proof}

\begin{cor}\label{rat}
Every rational smooth projective complex variety $X$ of dimension at most $4$ satisfies Property $\clubsuit^{(\mathrm{strong})}_{\id_X,K}$ for any finite field $K$ containing $\QQ(e^{2\ci \pi/6})$.
\end{cor}

\begin{proof}
Consider a weak factorization between $X$ and $\PP^n$.
All blow-up centers $Y_i$ have dimension at most two, and thus satisfy Property $\clubsuit^{\mathrm{strong}}_{\id_{Y_i},K}$ by Proposition \ref{two}.
Since $\PP^n$ satisfies Property $\clubsuit_{\id_{\PP^n},K}$, the claim follows from Corollary \ref{main} with $\QQ_\ext = \QQ(e^{2\ci \pi/6})$ and from  Proposition \ref{strong}.
\end{proof}

\subsection{Cubic fourfold}
Let $X$ be a cubic fourfold.
Its Hodge diamond decomposes into an ambient part, pulled back from $\PP^5$, and a primitive part.
We choose homogeneous basis $\mathcal{B}$ of $H^*(X,\QQ)$ and $\mathcal{H}$ of $H^*(X)^\Hdg$ that respect this direct sum.
The main theorem in \cite{KKPY} is the following.

\begin{thm}
A very general cubic fourfold $X$ is not rational.
\end{thm}

\begin{proof}
Let $X$ be a very general cubic fourfold, such that it has no Hodge classes in the primitive part of its cohomology.
Then $H^*(X)^\Hdg = H^*(X)^\mathrm{amb}$, which has dimension $5$.

Let $\zeta \in D(0,1)$ and $\ev \colon \widehat{R}^*(X,\QQ) \to S^*$ be the $\QQ$-evaluation map defined by
$$\ev(\mathfrak{q}')=b^{\deg(\mathfrak{q}')}~,~~\ev(Q)= \zeta^3 b^6 ~,~~ \textrm{and} \quad \ev(T_k) = 0,$$
for each $0 \leq k \leq h$.

Following Example \ref{Givental}, the matrix $\ev\(\kappa_\tau\)$ vanishes on the primitive part and has the following eigenvalues on the ambient part:
\begin{itemize}
\item $0$ with algebraic multiplicity $2$,
\item $9 \zeta b^2$, $9e^{2 \ci \pi/3} \zeta b^2$, $9e^{-2 \ci \pi/3}\zeta b^2$, each with algebraic multiplicity $1$.
\end{itemize}
Thus, for the eigenvalue $\alpha=0$, we obtain
$$\rho^X_{\ev, 0} = 2 \quad \textrm{and} \quad \nu^X_{\ev, 0}=1.$$
This contradicts Property $\clubsuit^{\mathrm{strong}}_{\id_X,\QQ}$.
By Remark \ref{incr}, it also contradicts Property $\clubsuit^{\mathrm{strong}}_{\id_X,K}$ for any number field $K$.
By Corollary \ref{rat}, we conclude that $X$ cannot be rational.
\end{proof}

\section{Examples regarding Property $\coeur$}
\textbf{Warning:} in this section, we have to pay attention to the difference between Properties $\coeur^{\mathrm{strong}}$ and $\coeur$.
However, we will show that they coincide at least up to dimension $2$.
\subsection{Basic cases}
\begin{pro}\label{basic}
Let $X$ be a point, a curve, a surface with vanishing Hodge number $h^{2,0}=0$, or a projective space $\PP^n$.
Then, for any number field $K$, the variety $X$ satisfies Property $\coeur^{\mathrm{strong}}_{\id_X,K}$.
\end{pro}

\begin{proof}
This follows from the fact that $h^{2,0}=0$ in all these cases; consequently, the integer $\nu^X_{\ev,\alpha}$ is always zero.
\end{proof}

\subsection{Surface with non-negative canonical class}
\begin{pro}\label{abel}
Every abelian surface $\Sigma$ satisfies Property $\coeur^{\mathrm{strong}}_{\id_{\Sigma},K}$ for any number field $K$.
\end{pro}

\begin{proof}
Since $K_\Sigma \geq 0$, we recall from the proof of Proposition \ref{nilp} (and Example \ref{Knef}) that for any $K$-evaluation map $\ev$, we have
\begin{eqnarray*}
\Sp^\Sigma_\ev & = & \left\lbrace \ev(T_0)b^{-2} \right\rbrace, \\
E^\Sigma_{\ev, \ev(T_0)b^{-2}} & = & H^*(\Sigma,K)_{S_K}.
\end{eqnarray*}
In particular, since $H^{1,0}(\Sigma,\CC) \neq 0$, it follows that ${\nu'}^\Sigma_{\ev,\ev(T_0)} \neq 0$.
\end{proof}

\begin{pro}
Every surface $\Sigma$ such that
$K_\Sigma \geq 0$ and $c_1(\Sigma) \neq 0$
satisfies Property $\coeur^{\mathrm{strong}}_{\id_{\Sigma},K}$ for any number field $K$.
\end{pro}

\begin{proof}
Since $K_\Sigma \geq 0$, we have
\begin{eqnarray*}
\Sp^\Sigma_\ev & = & \left\lbrace \ev(T_0)b^{-2} \right\rbrace, \\
E^\Sigma_{\ev, \ev(T_0)b^{-2}} & = & H^*(\Sigma,K)_{S_K},
\end{eqnarray*}
for any $K$-evaluation map $\ev \in \Omega_{\id_{\Sigma},K}$.
In Example \ref{Knef}, we computed the matrix $\left[\kappa_\tau-f_0\right]^{\mathcal{B}}_{\mathcal{B}}$, which takes the form
$$
\begin{pmatrix}
0&0&0&0&0\\
0&0&0&0&0\\
E&0&0&0&0\\
0&U&0&0&0\\
-T_h& 0&F &0&0
\end{pmatrix},
$$
where the blocks correspond to cohomology degrees and $U$ is a square matrix. Here,
$$E= \begin{pmatrix}
0\\
\vdots \\
0 \\
\delta_{c_1(\Sigma)\neq 0}\\
\end{pmatrix} ~, \quad \textrm{and} \quad
F= \begin{pmatrix}
\int_\Sigma c_1(\Sigma) \cup \phi_{N_1+1} & \dotsc & \int_\Sigma c_1(\Sigma) \cup \phi_{N_2} \\
\end{pmatrix}.
$$
Given the assumption $c_1(\Sigma) \neq 0$ and the fact that the pairing is perfect, the sub-matrices $E$ and $F$ are non-zero.
Since these matrices have coefficients in $K$, we obtain
$$\ev(E)=E \textrm{ and } \ev(F)=F.$$
Thus, the rank of $\ev\(\left[\kappa_\tau-f_0\right]^{\mathcal{B}}_{\mathcal{B}}\)$ is at least $2$ for any $K$-evaluation map.
Consequently, we obtain 
$$\gamma^\Sigma_{\ev,\ev(T_0)b^{-2}} = \gamma^\Sigma_{\ev-\ev(T_0)b^{-2}\un_\Sigma,0} \geq 2,$$
and $\Sigma$ satisfies Property $\coeur^{\mathrm{strong}}_{\id_{\Sigma},K}$.
\end{proof}

\begin{cor}\label{K3}
If a surface $\Sigma$ with $K_\Sigma \geq 0$ fails Property $\coeur^{\mathrm{strong}}_{\id_{\Sigma},K}$ for some number field $K$, then $\Sigma$ must satisfy
$$c_1(K_\Sigma)=0 ~,~~h^{2,0}(\Sigma) \neq 0 ~, ~~h^1(\Sigma,\CC)=0.$$
Furthermore, such a surface never satisfies Property $\coeur_{f,\QQ}$ for any morphism $f \colon \Sigma \to Z$ to some variety $Z$.
\end{cor}

\begin{proof}
If $\Sigma$ is a surface with $K_\Sigma \neq 0$ failing Property $\coeur^{\mathrm{strong}}_{\id_{\Sigma},K}$, the previous Proposition implies $c_1(K_\Sigma)=0$.
Furthermore, we must have $h^{2,0}(\Sigma) \neq 0$ and $h^1(\Sigma,\CC)=0$ because there is only one eigenvalue and we require $\nu \neq 0$ and $\nu'=0$.

Conversely, let $\Sigma$ be a surface with $c_1(K_\Sigma)=0$, $h^{2,0}(\Sigma) \neq 0$, and $h^1(\Sigma,\CC)=0$.
Since $c_1(\Sigma)=0$, the matrix $\left[\kappa_\tau-f_0\right]^{\mathcal{B}}_{\mathcal{B}}$ reduces to
$$
\begin{pmatrix}
0&0&0\\
0&0&0\\
-T_h&0&0
\end{pmatrix},
$$
where the blocks correspond to $H^0(\Sigma,\QQ)$, $H^2(\Sigma,\QQ)$, and $H^4(\Sigma,\QQ)$.
The matrix $\ev\(\left[\kappa_\tau-f_0\right]^{\mathcal{B}}_{\mathcal{B}}\)$ has rank at most one, depending on the value of $\ev(T_h)$.
Thus, for any morphism $f \colon \Sigma \to Z$ and any $\QQ$-evaluation map $\ev \colon \widehat{R}^*_{f,\QQ} \to S^*$, we find
$$\gamma^\Sigma_{\ev, \ev(T_0)b^{-2}} \leq 1 ~,~~ \nu^\Sigma_{\ev, \ev(T_0)b^{-2}}=h^{2,0}(\Sigma) = 1 ~,~~ {\nu'}^\Sigma_{\ev, \ev(T_0)b^{-2}}=h^{1,0}(\Sigma)=0,$$
contradicting Property $\coeur_{f,\QQ}$.
\end{proof}

\begin{rem}\label{K3rem}
Let $\Sigma$ be a minimal surface with
$$c_1(K_\Sigma)=0 ~,~~h^{2,0}(\Sigma) \neq 0 ~, ~~h^1(\Sigma,\CC)=0.$$
The Kodaira dimension of $\Sigma$ cannot be $-\infty$ (since $h^{2,0}(\Sigma) \neq 0$) or $2$ (since $c_1(K_\Sigma)=0$).
If the Kodaira dimension of $\Sigma$ is $0$, $\Sigma$ must be a K3 surface because $h^{2,0}(\Sigma) \neq 0$ and $h^1(\Sigma,\CC)=0$.

If the Kodaira dimension of $\Sigma$ is $1$, $\Sigma$ is an elliptic surface $f \colon \Sigma \to B$ over a smooth curve $B$.
%
By the Leray spectral sequence, $h^1(B,\CC) \neq 0$ implies $h^1(\Sigma,\CC) \neq 0$, so $B = \PP^1$.
Furthermore, an elliptic surface over $\PP^1$ with $h^{2,0}=1$ and $c_1=0$ is necessarily a K3 elliptic surface.
\end{rem}

\subsection{Other surfaces}
\begin{pro}\label{mini}
	Let $K$ be a number field containing $\QQ(\ci)$.
	A surface $\Sigma$ fails Property $\coeur_{\id_{\Sigma},K}$ if and only if its minimal model $\Sigma'$ satisfies
	$$c_1(K_{\Sigma'})=0 ~,~~h^{2,0}(\Sigma') \neq 0 ~, ~~h^1(\Sigma',\CC)=0.$$
\end{pro}

\begin{proof}
	This follows from Corollaries \ref{K3} and \ref{main} (where $\QQ_\ext=\QQ(\ci)$).
	By choosing a weak factorization between $\Sigma$ and its minimal model $\Sigma'$, we observe that all blow-up centers satisfy Property $\coeur^{\mathrm{strong}}_{\id_{\pt},K}$, because they are points.
\end{proof}

\begin{pro}\label{strongsurface}
Let $X$ be a variety of dimension at most two and $K$ be a number field containing $\QQ(\ci)$.
Then we have
$$X ~\textrm{satisfies Property}~ \coeur_{\id_X,K} \iff X~\textrm{satisfies Property}~\coeur^{\mathrm{strong}}_{\id_X,K}.$$
\end{pro}

\begin{proof}
The case where $h^{2,0}(X)=0$ has been treated in Proposition \ref{basic}, so that we can assume that $X$ is a surface with $h^{2,0}(X) \neq 0$.
Furthermore, we assume that $X$ satisfies Property $\coeur_{\id_X,K}$, as the converse implication is clear.

Denote by $\Sigma$ the minimal model of $X$.
By Proposition \ref{mini}, it is a surface with $K_\Sigma$ nef and with either $c_1(K_\Sigma) \neq 0$ or $h^1(\Sigma,\CC) \neq 0$ (possibly both).
Furthermore, $X$ is obtained from $\Sigma$ by blowing-up points.
Denote by $\mathrm{pr} \colon X \to \Sigma$ this birational map.

First, we assume that $h^1(\Sigma,\CC) \neq 0$.
Let $\ev_0 \in \Omega_{\id_X,K}$ be any $K$-evaluation map that we may assume to be normalized.
We aim to show that there is a unique eigenvalue $\alpha_0$ whose generalized eigenspace contains the subspaces $H^{2,0}(X) \simeq H^{2,0}(\Sigma)$ and $H^1(X) \simeq H^1(\Sigma)$, because then $\ev_0$ would satisfy Property $\coeur_{\id_X,K}$.
Consider the one-parameter deformation $(\ev_{\zeta})_{\zeta \in D(0,1)} \in \Omega_{\id_X,K}^{D(0,1)}$ introduced in the proof of Proposition \ref{incr}, with $\ev_0=\ev_{\zeta_0}$ for some $\zeta_0 \in F_K$.
Fix $r<1$ large enough so that $\zeta_0 \in D(0,r)$ and the subset
$$\mathcal{U} = \left\lbrace \zeta \in D(0,r) ~|~~\textrm{the inequality in Proposition \ref{incr3} is satisfied for $\ev_{\zeta}$} \right\rbrace$$
is infinite.
Hence, for any $\zeta \in \mathcal{U}$, the $K$-evaluation map $\ev_{\zeta}$ induces a $K$-evaluation map in $\Omega_{\mathrm{pr},K}$.
By Corollary \ref{equality2}, we have an induced $\ev'_\zeta \in \Omega_{\id_\Sigma,K}$ and a decomposition of the generalized eigenspaces of $\ev_{\zeta}(\kappa^X_\tau)$ as a direct sum of generalized eigenspaces from $\ev'_{\zeta}(\kappa^\Sigma_\tau)$ and from $\kappa^\pt_\tau$.
Since the subspaces $H^{2,0}(\Sigma)$ and $H^1(\Sigma)$ are contained in a single generalized eigenspace of $\ev'_{\zeta}(\kappa^\Sigma_\tau)$, then the same holds true for $\ev_{\zeta}(\kappa^X_\tau)$.
As the parameter $\zeta$ moves outside of $\mathcal{U}$, the generalized eigenspaces of $\ev_{\zeta}(\kappa^X_\tau)$ can only merge, so that the subspaces $H^{2,0}(X)$ and $H^1(X)$ are still contained in a single one.

Second, we assume that $h^1(\Sigma,\CC)=0$, and thus $c_1(K_\Sigma) \neq 0$.
Let $\ev \in \Omega_{\mathrm{pr},K}$ be a $K$-evaluation map and denote by $\alpha_\ev \in \Sp^X_\ev$ the unique eigenvalue whose generalized eigenspace contains the subspace $H^{2,0}(X) \simeq H^{2,0}(\Sigma)$.
By Corollary \ref{equality2}, we obtain from $\ev$ a corresponding $K$-evaluation map $\ev' \in \Omega_{\id_\Sigma,K}$ and we have
$f \colon \EE^\Sigma_{\ev',\alpha} \hookrightarrow \EE^X_{\ev,\alpha}$, with the commutation relation
$$\ev\(\kappa^X_\tau\) \circ f = f \circ \ev'\(\kappa^\Sigma_\tau\)$$
as in the end of Remark \ref{equality3}.
We recall that the matrix of $\ev'\(\kappa^\Sigma_\tau\)-\ev'(T_0)$ obtained from Example \ref{Knef} takes the form
$$
\begin{pmatrix}
0&0&0&0&0\\
0&0&0&0&0\\
E&0&0&0&0\\
0&0&0&0&0\\
-\ev'(T_h)& 0&F &0&0
\end{pmatrix},
$$
where the rows and columns corresponding to $H^1(\Sigma)$ and $H^3(\Sigma)$ do not exist, but we leave them here to be consistent with Example \ref{Knef}.
We observe that the matrices $E$ and $F$ have coefficients in $\QQ$ and are non-zero, as $c_1(K_\Sigma) \neq 0$.
In particular, there is a $2 \times 2$ minor $M' \in \mathrm{GL}_2(F_K)$ of the above matrix whose determinant is a non-zero rational number and is independant of the choice of the $K$-evaluation map $\ev'$.
Similarly, since the matrix $[\ev(\kappa^X_{\tau})]-\alpha_\ev$ (written in the basis given by the isomorphism $\Psi$ from Theorem \ref{Bl}) is diagonal by block with blocks of size $1$ corresponding to the blow-up centers, i.e.~to points, and one block corresponding to $\Sigma$, there is a $2 \times 2$ minor $M \in \mathrm{GL}_2(F_K)$ of the matrix $[\ev(\kappa^X_{\tau})]-\alpha_\ev$ whose determinant is a non-zero rational number and is independant of the choice of $\ev$.
Here, we recall that $\alpha_\ev$ denotes the eigenvalue of the block corresponding to $\Sigma$.

Now, let $\ev_0 \in \Omega_{\id_X,K}$ that we may assume to be normalized. Consider again the one-parameter deformation $(\ev_{\zeta})_{\zeta \in D(0,1)} \in \Omega_{\id_X,K}^{D(0,1)}$introduced in the proof of Proposition \ref{incr}, with $\ev_0=\ev_{\zeta_0}$ for some $\zeta_0 \in F_K$.
Fix $r<1$ large enough so that $\zeta_0 \in D(0,r)$ and the subset
$$\mathcal{U} = \left\lbrace \zeta \in D(0,r) ~|~~\textrm{the inequality in Proposition \ref{incr3} is satisfied for $\ev_{\zeta}$} \right\rbrace$$
is infinite.
Hence, for any $\zeta \in \mathcal{U}$, the $K$-evaluation map $\ev_{\zeta}$ induces a $K$-evaluation map in $\Omega_{\mathrm{pr},K}$.
Consequently, the matrix $[\ev_\zeta(\kappa_{\tau})]-\alpha_{\ev_{\zeta}}$ contains a $2 \times 2$ minor $M \in \mathrm{GL}_2(F_K)$ whose determinant is a non-zero rational number which is constant in $\zeta \in \mathcal{U}$.
But the entries of the matrix $[\ev_\zeta(\kappa_{\tau})]-\alpha_{\ev_{\zeta}}$ are locally power series in $\zeta \in D(0,r)-\mathcal{S}$, for some finite subset $\mathcal{S}$.
Since the subset $\mathcal{U}$ is infinite and the determinant of the minor $M$ is a non-zero constant on $\mathcal{U}$, then it stays constant in $D(0,r)-\mathcal{S}$, as well as on $D(0,r)$ by continuity. In particular, we get
$$\gamma^X_{\ev_0,\alpha_{\ev_0}} = \gamma^X_{\ev_{\zeta_0},\alpha_{\ev_0}} \geq 2,$$
where we recall that the eigenvalue $\alpha_{\ev_0}$ is the one of the block corresponding to $\Sigma$, and thus containing the subspace $H^{(2)}(X)$.
\end{proof}

\begin{pro}\label{mini2}
Let $\Sigma$ be a surface whose minimal model $\Sigma'$ satisfies
$$c_1(K_{\Sigma'})=0 ~,~~h^{2,0}(\Sigma') \neq 0 ~, ~~h^1(\Sigma',\CC)=0.$$
Let $\ev$ be a $\QQ(\ci)$-evaluation map and $\alpha \in \Sp^\Sigma_\ev$ be such that
$$\gamma^\Sigma_{\ev,\alpha} <2 ~, \quad \nu^\Sigma_{\ev,\alpha} \neq 0 \quad \textrm{and} \quad  {\nu'}^\Sigma_{\ev,\alpha} = 0.$$
Then there exist a number field extension $\QQ(\ci) \subset \widetilde{\QQ}(\ci)$ and an embedding of free $S^\mathrm{even}_{\widetilde{\QQ}(\ci)}$-modules
$$H^*(\Sigma',\widetilde{\QQ}(\ci))_{S_{\widetilde{\QQ}(\ci)}} \hookrightarrow E^\Sigma_{\ev,\alpha},$$
that is in the linear span of Hodge structures from $H^*(\Sigma',\QQ)_S$ to $H^*(\Sigma,\QQ)_S$.
\end{pro}

\begin{proof}
Consider the number field $\QQ_\ext = \QQ(\ci)$ as in Corollary \ref{main}, and let
$$\Sigma=:X_0 \leftrightarrow X_1 \leftrightarrow X_2 \leftrightarrow X_3 \leftrightarrow \dotsc \leftrightarrow X_q:=\Sigma'$$
be a weak factorization between $\Sigma$ and its minimal model $\Sigma'$ with blow-up centers $Y_1, \dotsc, Y_q$.
We define the finite field extension $\QQ(\ci) \subset \widetilde{\QQ}(\ci)$ by adding all the roots of the characteristic polynomials of the matrices involved in the weak factorization.
From Remark \ref{equality3}, we have the equality
$$\nu^\Sigma_{\ev,\alpha} = \nu^{\Sigma'}_{\ev_q,\alpha} + \sum_{i,k} \nu^{Y_i}_{\ev_{i,k},\alpha}.$$
Since all blow-up centers are points, they have no $h^{2,0}$, so $\nu^{Y_i}_{\ev_{i,k},\alpha} =0$.
Therefore, $\nu^{\Sigma'}_{\ev_q,\alpha} \neq 0$, which implies $\alpha \in \Sp^\Sigma_\ev \cap \Sp^{\Sigma'}_{\ev_q}$.
By Remark \ref{equality3}, there exist morphisms of free $S^{\mathrm{even}}$-modules and of Hodge structures
$f_k \colon H^*(\Sigma',\QQ)_S \to H^*(\Sigma,\QQ)_S~,~~1 \leq k \leq N,$
such that
$$\(s_1f_1+\dotsb+s_Nf_N\) \colon H^*(\Sigma',\widetilde{\QQ}(\ci))_{S_{\widetilde{\QQ}(\ci)}} = E^{\Sigma'}_{\ev_q,\alpha} \hookrightarrow E^\Sigma_{\ev,\alpha},$$
where $(s_1, \dotsc, s_N)$ is a $\QQ$-basis of the number field $\widetilde{\QQ}(\ci)$.
\end{proof}

\subsection{Cubic fourfolds}
Let $X$ be a cubic fourfold.
Its Hodge diamond decomposes into an ambient part, pulled back from $\PP^5$, and a primitive part.

\begin{pro}\label{no coeur}
A cubic fourfold $X$ fails Property $\coeur_{\id_X,\QQ}$.
\end{pro}

\begin{proof}
Let $X$ be a cubic fourfold, and let $\zeta \in D(0,1)$.
We define $\ev_0 \colon R^*_{\id_X,\QQ} \to S^*$ be the $\QQ$-evaluation map
$$\ev_0(\mathfrak{q}')=b^{\deg(\mathfrak{q}')}~,~~\ev_0(Q)= \zeta^3 b^6 ~,~~ \textrm{and} \quad \ev_0(T_k) = 0,$$
for each $0 \leq k \leq h$.

Following Example \ref{Givental}, we deduce that the endomorphism $\ev_0\(\kappa_\tau\)$ vanishes on the primitive part and has the following eigenvalues on the ambient part:
\begin{itemize}
\item $0$ with algebraic multiplicity $2$,
\item $9 \zeta b^2$, $9e^{2 \ci \pi/3} \zeta b^2$, $9e^{-2 \ci \pi/3} \zeta b^2$, each with algebraic multiplicity $1$.
\end{itemize}
Furthermore, the matrix representing the action on the ambient part of the generalized eigenspace for the eigenvalue 0 is conjugate over $\QQ$ to the Jordan block
$$\begin{pmatrix}
0 & 1 \\
0 & 0
\end{pmatrix}.
$$
Since this submatrix has coefficient in $\QQ$, its evaluation is simply the matrix itself, which has rank one.
Consequently, for the eigenvalue $\alpha=0$, we obtain
$$\gamma^X_{\ev_0, 0} = 1 ~, \quad \nu^X_{\ev_0, 0}=1~, \quad \textrm{and} \quad {\nu'}^X_{\ev_0, 0}=0.$$

More generally, let $\ev \in U_{\id_X,\QQ}$.
By irreducibility of the monodromy representation on the primitive cohomology of $X$, we know that the matrix of $\ev(\kappa_\tau)$ is still diagonal by blocks (one for the ambient part and one for the primitive part) and that the primitive block is a multiple of the identity matrix, say $\alpha_0 \cdot \id$ with $\alpha_0 \in F_K$. In particular, the rank of the matrix $\ev(\kappa_\tau)-\alpha_0$ on the primitive block is zero.

On the ambient part, since the cardinality of $\Sp^X_\ev$ is assumed to be maximal (we have $\ev \in U_{\id_X,\QQ}$), we have at least $4$ distinct eigenvalues, so that the generalized eigenspaces have dimension at most $2$.
Consequently, for any $\alpha \in \Sp^X_\ev$, we have $\gamma^X_{\ev,\alpha} \leq 1$.
This contradicts Property $\coeur_{\id_X,\QQ}$.
\end{proof}

\begin{thm}\label{rat cubic K3}
If $X$ is a rational cubic fourfold, then there exists a projective K3 surface $\Sigma$ and an isomorphism of Hodge structures
$$H^4(X,\QQ)_\mathrm{primitive} \simeq H^2(\Sigma,\QQ)(-1).$$
\end{thm}

\begin{proof}
Consider a weak factorization
$$X=:X_0 \leftrightarrow X_1 \leftrightarrow X_2 \leftrightarrow X_3 \leftrightarrow \dotsc \leftrightarrow X_q:=\PP^4,$$
with blow-up centers $Y_1,\dotsc,Y_q$ as in Definition \ref{weak}, and let $\QQ_\ext := \QQ(e^{2\ci \pi/6})$ be the number field relevant for Corollary \ref{main}.

Let $\zeta \in D(0,1)$ and $\ev_0$ be the $\QQ$-evaluation map from the proof of Proposition \ref{no coeur}.
We take $\alpha=0 \in \Sp^X_{\ev_0}$, so that
$$\gamma^X_{\ev_0, 0} = 1~, \quad \nu^X_{\ev_0, 0}=1~, \quad \textrm{and} \quad {\nu'}^X_{\ev_0, 0}=0.$$
Following Remark \ref{equality3}, we choose a number field extension $\widetilde{\QQ}_\ext$ large enough so that, for every $\QQ_\ext$-evaluation map $\ev$ involved in the weak factorization, the characteristic polynomial of $\ev(\kappa_\tau)$ is split.
Take a $\QQ$-basis $(s_1,\dotsc, s_N)$ of $\widetilde{\QQ}_\ext$.
Then by Remark \ref{equality3} we have
$$\epsilon^X_{\ev_0,\alpha} = \epsilon^{\PP^4}_{\ev_q,\alpha} + \sum_{(i,k) \in I} \epsilon^{Y_i}_{\ev_{i,k},\alpha}~, \quad \epsilon \in \left\lbrace \gamma, \nu, \nu'\right\rbrace.$$
Since $H^{2,0}(\PP^4)=0$, we have $\nu^{\PP^4}_{\ev_q,\alpha}=0$.
Given $\nu^X_{\ev, \alpha}=1$, there must exist a pair $(i,k) \in I$ such that $\nu^{Y_i}_{\ev_{i,k},\alpha} = 1$.
This implies
$\alpha \in \Sp^X_{\ev_0} \cap \Sp^{Y_i}_{\ev_{i,k}}$.
Furthermore, we must have
$${\nu'}^{Y_i}_{\ev_{i,k},\alpha} = 0 \quad \textrm{and} \quad \gamma^{Y_i}_{\ev_{i,k},\alpha} <2,$$
because it is true for $X$, so $Y_i$ fails Property $\coeur^{\mathrm{strong}}_{\id_{Y_i},\QQ_\ext}$.
By Propositions \ref{basic} and \ref{mini}, $Y_i$ is a surface whose minimal model $\Sigma$ satisfies $c_1(K_\Sigma)=0$, $h^{2,0}(\Sigma) \neq 0$, and $h^1(\Sigma,\CC)=0$.

Combining Proposition \ref{mini2} and Remark \ref{equality3}, we obtain $N$ morphisms of $S$-modules and of Hodge structures
$$f_1, \dotsc, f_N \colon H^*(\Sigma,\QQ)_S(-1) \to H^{*+2}(X,\QQ)_S$$
such that the morphism
$$f := s_1f_1+ \dotsb + s_N f_N$$
is the composition of embeddings
$$H^*(\Sigma,\widetilde{\QQ}_\ext)_{S_{\widetilde{\QQ}_\ext}} \hookrightarrow E^{Y_i}_{\ev_{i,k},\alpha} \hookrightarrow E^X_{\ev_0,\alpha} \subset H^{*+2}(X,\widetilde{\QQ}_\ext)_{S_{\widetilde{\QQ}_\ext}}.$$
Furthermore, $\Sigma$ is a K3 surface because $h^{3,1}(X)=1$ and thus $h^{2,0}(\Sigma) = 1$, see Remark \ref{K3rem}.
Since $\ev_0$ is defined over $\QQ$, Example \ref{Givental} implies
$$E^X_{\ev_0,\alpha} = P^*_S \oplus H^*(X,\QQ)_{\mathrm{primitive},S},$$
where we set
$$w_3 := H^3-21 a^3b^6~,~~w_4 := H^4-6 a^3b^6H~,~~P^*_S:= S \cdot \langle w_3, w_4 \rangle \subset H^*(X,\QQ)_S,$$
with $H$ the hyperplane class in $X$.

\begin{lem}\label{sub}
For each $1 \leq k \leq N$, the morphism $f_k$ of $S$-modules and of Hodge structures satisfies
$$f_k \colon H^*(\Sigma,\QQ)_S(-1) \to P^{*+2}_S \oplus H^{*+2}(X,\QQ)_{\mathrm{primitive},S}.$$
\end{lem}

\begin{proof}
Let $\bar{f}, \bar{f}_1, \dotsc, \bar{f}_N$ be the projections of $f, f_1, \dotsc, f_N$ onto the ambient cohomology of $X$.
Recall that the image of $\bar{f}$ lands in $P^*_{S_{\widetilde{\QQ}_\ext}}$.
We aim to show that for each $1 \leq k \leq N$ the image of $\bar{f}_k$ lies in $P^*_S$.

Consider the basis $\(H^0, \dotsc, H^4\)$ of $H^*(X,\QQ)_{\mathrm{ambient}}$ and a graded basis $\(e_1,\dotsc, e_m\)$ of $H^*(\Sigma,\QQ)$.
We write
$$\bar{f}_k(e_j) = \sum_{i=0}^4 \zeta^{(k)}_{i,j} b^{\deg(e_j)+2-2i} H^i ~,~~ \zeta^{(k)}_{i,j} \in F.$$
Recalling that $\bar{f} = s_1 \bar{f}_1 + \dotsb + s_N \bar{f}_N$ and that the image of $\bar{f}$ lies in $P^*_{S_{\widetilde{\QQ}_\ext}}$, then for each $1 \leq j \leq m$, there exist $x_j,y_j \in S_{\widetilde{\QQ}_\ext}$ such that
$$\sum_{\substack{1 \leq k \leq N \\ 0 \leq i \leq 4}} s_k \zeta^{(k)}_{i,j} b^{\deg(e_j)+2-2i} H^i = x_j w_3 + y_j w_4.$$
By identifying the components of $H^0, \dotsc, H^4$ on both sides, we obtain
\begin{eqnarray*}
\sum_{1 \leq k \leq N} s_k \zeta^{(k)}_{0,j} b^{\deg(e_j)+2} & = & -21 a^3b^6 x_j, \\
\sum_{1 \leq k \leq N} s_k \zeta^{(k)}_{1,j} b^{\deg(e_j)} & = & -6 a^3b^6y_j, \\
\sum_{1 \leq k \leq N} s_k \zeta^{(k)}_{2,j} b^{\deg(e_j)-2} & = & 0, \\
\sum_{1 \leq k \leq N} s_k \zeta^{(k)}_{3,j} b^{\deg(e_j)-4} & = & x_j, \\
\sum_{1 \leq k \leq N} s_k \zeta^{(k)}_{4,j} b^{\deg(e_j)-6} & = & y_j.
\end{eqnarray*}
By combining the first and the fourth equations, and the second and the fifth equations, and using that $\(s_1, \dotsc, s_N\)$ is a basis of $\widetilde{\QQ}_\ext$ over $\QQ$, we obtain
\begin{eqnarray*}
\(21 a^3 \zeta^{(k)}_{3,j} + \zeta^{(k)}_{0,j}\) & = & 0, \\
\(6 a^3 \zeta^{(k)}_{4,j} + \zeta^{(k)}_{1,j}\) & = & 0, \\
\zeta^{(k)}_{2,j} & = & 0, \\
\end{eqnarray*}
for each $1 \leq k \leq N$.
These relations confirm that each $\bar{f}_k$ maps into $P^*_S$, as desired.
\end{proof}

Lemma \ref{sub} implies that we obtain a morphism of Hodge structures
$$f_k \colon \langle \un_\Sigma b^2, [\pt] b^{-2} \rangle_F(-1) \oplus H^2(\Sigma,\QQ)_F(-1) \to \langle w_3b^{-2},w_4b^{-4} \rangle_F \oplus H^4(X,\QQ)_{\mathrm{primitive},F}.$$
Our goal is to eliminate the extra copies of $\QQ$ and the non-archimedean field $F$.
Recall that the $\QQ_\ext$-evaluation map $\ev_{i,k}$ for $Y_i$ corresponds to the evaluation map $\ev_0$ for $X$ under the change of variables along the weak factorization.
Similarly, there exists a $\QQ_\ext$-evaluation map $\ev'$ for $\Sigma$ corresponding to the evaluation map $\ev_{i,k}$ for $Y_i$ under the change of variables along a weak factorization between $Y_i$ and $\Sigma$.
By Remark \ref{equality3}, we have the commutation relation
$$\ev_0\(\kappa^X_\tau\) \circ f = f \circ \ev'\(\kappa^\Sigma_\tau\).$$
As the dimensions match, the embedding
$$f \colon \langle \un_\Sigma b^2, [\pt] b^{-2} \rangle_{F_{\widetilde{\QQ}_\ext}} \oplus H^2(\Sigma,\QQ)_{F_{\widetilde{\QQ}_\ext}} \hookrightarrow \langle w_3b^{-2},w_4b^{-4} \rangle_{F_{\widetilde{\QQ}_\ext}} \oplus H^4(X,\QQ)_{\mathrm{primitive},F_{\widetilde{\QQ}_\ext}}$$
is an isomorphism of $F_{\widetilde{\QQ}_\ext}$-vector spaces.
Consequently, the endomorphisms $\ev_0\(\kappa^X_\tau\)$ and $\ev'\(\kappa^\Sigma_\tau\)$ are conjugate and share the same rank.

From Example \ref{Givental}, the endomorphism $\ev_0\(\kappa^X_\tau\)$ vanishes on the primitive part of the cohomology and satisfies
$$\ev_0\(\kappa^X_\tau\)(x w_3+ y w_4) = x w_4~,~~ \forall ~x,y \in \QQ.$$
From the proof of Corollary \ref{K3}, the endomorphism $\ev'\(\kappa^\Sigma_\tau\)$ vanishes except for
$$\ev'\(\kappa^\Sigma_\tau\)(\un_\Sigma) = -\ev'(T_p) [\pt].$$
By homogeneity of $\ev'$, we must have $\ev'(T_p) = \zeta \cdot b^{-2}$ for some $\zeta \neq 0 \in F_{\QQ_\ext}$.

Let $\(u_2,\dotsc, u_{p-1}\)$ of $H^2(\Sigma,\QQ)$ be a basis of $H^2(\Sigma,\QQ)$, and set $u_1:=\un_\Sigma b^2$, $u_p:=[\pt]b^{-2}$.
For each $1 \leq k \leq N$, we write
$$f_k(u_j) - \(x^{(k)}_j w_3b^{-2} + y^{(k)}_j w_4b^{-4}\) \in H^4(X,\QQ)_{\mathrm{primitive},F}~,~~x^{(k)}_j,y^{(k)}_j \in F.$$
Fixing a basis $\(v_1,\dotsc,v_q\)$ of $H^4(X,\QQ)_{\mathrm{primitive}}$, the conjugation relation implies that the matrix $A_k$ of $f_k$, relative to the bases $\(u_1,\dotsc, u_p\)$ and $\(w_3b^{-2},v_1,\dotsc,v_q,w_4b^{-4}\)$, takes the form
$$A_k := \begin{pmatrix}
x_1^{(k)} & 0 \dotsc 0 & 0 \\
\star & A'_k & 0 \\
\star & \star \dotsc \star & y_p^{(k)}
\end{pmatrix},$$
where $\sum_{k=1}^N s_k x_1^{(k)} = -\zeta \sum_{k=1}^N s_k y_p^{(k)}$
and $A'_k$ is the submatrix with coefficients in $F$ corresponding to the restriction
$$f'_k \colon H^2(\Sigma,\QQ)_F(-1) \to H^4(X,\QQ)_{\mathrm{primitive},F}.$$
The matrix of the isomorphism $f$, relative to the above bases, is given by $A=s_1 A_1 + \dotsb + s_N A_N$.
Since $A$ is invertible and lower block-triangular, the diagonal block $A'=s_1 A'_1 + \dotsb + s_N A'_N$ must also be invertible.

Choose an integer $L \in \NN$ such that, for each $1 \leq k \leq N$, the matrix
$$B'_k:=\(a^L \cdot A'_k\) \in \mathrm{Mat}\(\QQ[[a^{\QQ^*_+}]]\)$$
has positive valuation.
Let $\mathcal{Q}$ be the set of exponents appearing in $B'_1, \dotsc, B'_N$:
$$\mathcal{Q} = \left\lbrace q \in \QQ ~|~~\exists 1 \leq k \leq N ~\mathrm{st}~~\mathrm{Coeff}_{a^q}\(B'_k\) \neq 0 \right\rbrace \subset \QQ^*_+.$$
We write
$$B'_k = \sum_{l \in \mathcal{Q}} B'_{k,l} a^l \quad \textrm{with $B'_{k,l} \in \mathrm{Mat}(\QQ)$},$$
and denote by $g'_{k,l}$ the morphism of $\QQ$-vector spaces whose matrix in the above basis is $B'_{k,l}$.
Since the Hodge structure on $H^2(\Sigma,\QQ)_F(-1)$ (resp.~$H^4(X,\QQ)_{\mathrm{primitive},F}$) is obtained from the Hodge structure on $H^2(\Sigma,\QQ)(-1)$ (resp.~$H^4(X,\QQ)_{\mathrm{primitive}}$) by extension of scalars from $\QQ$ to $F$, the linear maps
$$g'_{k,l} \colon H^2(\Sigma,\QQ)(-1) \to H^4(X,\QQ)_{\mathrm{primitive}}$$
are morphisms of Hodge structures.
We seek a $\QQ$-linear combination of these maps that constitutes an isomorphism.

For any $M \geq 0$, fix a positive integer $A_M \in \NN^*$ such that the finite set $\mathcal{Q} \cap \left[ 0,M \right]$ is contained in $\left\lbrace \frac{p}{A_M} ~|~~1 \leq p \leq MA_M \right\rbrace$.
Suppose that for every $M \geq 0$ and every $\(q_1, \dotsc, q_N\) \in \QQ^N$, we have
$$\det\( \sum_{l=1}^{M A_M} q_1^l B'_{1,l} + \dotsb + q_N^l B'_{N,l} \)=0.$$
This would imply that the polynomial
$$\det\( \sum_{l=1}^{MA_M} X_1^l B'_{1,l} + \dotsb + X_N^l B'_{N,l} \)=0$$
is identically zero.
Consequently, for any $M \geq 0$, the following truncation would vanish:
$$\det\( \sum_{l=1}^{M A_M} \(s_1^{1/l}a^{1/A_M}\)^l B'_{1,l} + \dotsb + \(s_N^{1/l} a^{1/A_M}\)^l B'_{N,l} \)=0.$$
Let $D \geq 0$ be the valuation of the non-zero determinant
$$\det \( s_1 B'_1+ \dotsb + s_N B'_N \) \neq 0 \in a^D\cdot \QQ_\ext[[a^{\QQ_+}]].$$
For $M>D$, the vanishing of the truncation leads to a contradiction.
Therefore, there must exist some $M \geq 0$ and $\(q_1, \dotsc, q_N\) \in \QQ^N$ such that
$$\det\( \sum_{l=1}^{MA_M} q_1^l B'_{1,l} + \dotsb + q_N^l B'_{N,l} \) \neq 0.$$
The resulting $\QQ$-linear combination
$$\sum_{l=1}^{MA_M} q_1^l g'_{1,l} + \dotsb + q_N^l g'_{N,l}$$
is thus an isomorphism of Hodge structures, where $\Sigma$ is a projective $K3$ surface.
\end{proof}

\end{document}